\documentclass[12pt]{amsart}
\usepackage{graphicx}
\usepackage[english]{babel}

\usepackage[utf8]{inputenc}
\usepackage{amsmath,amsfonts,amssymb,amscd,amsthm,txfonts,hyperref}
\usepackage[all]{xy}
\usepackage{multibib}

\usepackage{newunicodechar}

\newunicodechar{≥}{\ensuremath{\geq}}

\newtheorem{theorem}{Theorem}[section]
\newtheorem{proposition}{Proposition}[section]
\newtheorem{lemma}{Lemma}[section]
\newtheorem{corollary}{Corollary}[section]
\newtheorem{remark}{Remark}[section]

\theoremstyle{definition}
\newtheorem{definition}{Definition}[section]

\newtheorem{notation}{Notation}[section]

\newcommand{\C}{\varmathbb C}
\newcommand{\R}{\varmathbb R}
\renewcommand{\S}{\varmathbb S}
\newcommand{\Z}{\varmathbb Z}

\newcommand{\N}{\varmathbb N}

\newcommand{\Lop}{\mathcal L}
\newcommand{\Rop}{\mathcal R}
\newcommand{\Qop}{\mathcal Q}
\newcommand{\E}{\varmathbb E}

\renewcommand{\Re}{\mbox{Re}}

\title[Scattering for quinti Hartree]{A scattering result around a non-localised equilibria for the quintic Hartree equation for random fields}
\author{Cyril Malézé}
\address{Centre Mathématique Laurent Schwartz\\
\'Ecole Polytechnique, CNRS, Université Paris-Saclay\\
Palaiseau, 91128 Cedex, France}
\email{cyril.maleze@polytechnique.edu}

\begin{document}

\maketitle

\begin{abstract}
    We consider a quintic Hartree equation for a random field, which describes the temporal evolution of a infinitely many fermions, considering a three body interaction. We show a scattering result around a non-localised equilibria of the equation, for high dimensions $d\geq 4$. The Hartree equation for random variables was introduced by de Suzzoni \cite{SLSEDP} but only for a two body interaction, that leads to a cubic Hartree equation for random variables. Scattering results for the cubic Hartree equation have been shown in \cite{CodS18,CodS20}, and we extend those results to the quintic Hartree equation. As in \cite{CodS20} we consider a large range of potentials that includes the Dirac delta
\end{abstract}

\section{Introduction}

In this paper, we prove the stability of a non-localised equilibrium for an equation for random fields, which represents the dynamics of an infinite number of fermions, considering a three body interaction. Precisely, we prove a scattering result concerning the following dispersive equation:

\begin{equation}\label{eqprinc}
    i\partial_t X =-\Delta X + w\circledast(\rho,\rho) X,
\end{equation}

where $X:\R\times\R^d\times\Omega\to \C$ is a time dependent random field over $\R^d$, and $(\Omega,\mathcal{A},d\omega)$ is the underlying probability space, where we denote $\rho:=\E[\lvert X\rvert^2]$ the density function and:$$w\circledast(\rho,\rho):=\int w(x_1-\cdot,x_2-\cdot)\rho(x_1)\rho(x_2)dx_1dx_2.$$ The expectation is with respect to the probability space: $\E(\lvert X\rvert^2)(t,x):=\int_\Omega \lvert X\rvert^2(t,x,\omega)d\omega$.

In this work, we extend a result proved by Collot and de Suzzoni in \cite{CodS18}, in which the authors prove a result of scattering for a similar equation than (\ref{eqprinc}), but for a two body interaction. 

Lewin and Sabin were the first to investigate the problem of stability of equilibria for the Hartree equation in a density operator framework \cite{lewsabII,lewsabI}. The authors proved a scattering results around Fourier multipliers, that are translation invariant equilibria of the Hartree equation. Important tools in their works are dispersion estimates for orthonormal systems \cite{StrWaveOp,StrWaveOp2}.  The work of Lewin and Sabin was continued by Chen, Hong and Pavlović who proved global well-posedness in dimensions $d = 2,\ 3$ for an equilibrium solution at zero temperature \cite{chen2017global} and scattering results in the case of dimension 3 and higher \cite{chen2018scattering}, left undone in \cite{lewsabII}. 
Solutions of the Hartree equation for the density operators with an infinite number of particles were also studied in previous works such as \cite{Bove1974Jan,bove1974existence,bove1976hartree,chadam1976time}. More recently, Lewin and Sabin studied the semi-classical limit of the Hartree equation\cite{Sabin}, and Pusateri and Sigal described the long-time behaviour of the time-dependent Kohn-Sham equation \cite{Pusateri_2021}, which is closely related to the Hartree equation.

The Hartree equation for random field was introduced in the work of de Suzzoni \cite{SLSEDP}, and studied by Collot and de Suzzoni in \cite{CodS18,CodS20}. In \cite{SLSEDP}, the author links the equilibrium studied in \cite{CodS18} and in \cite{CodS20} for the pairwise potential equation for random fields, to the translation invariant stationary states for the analogous equation in density operators framework, studied in \cite{lewsabI,lewsabII}. The stability of this non-localised equilibrium was showed through the proof of a scattering result, first in dimension higher than $4$ \cite{CodS18} and then in dimension $2$ and $3$ \cite{CodS20}. Many thermodynamical equilibrias are included in the result of Collot and de Suzzoni (for instance for bosonic gazes at positive temperature), but not the Fermi gas at zero temperature. In \cite{NAKASTUD}, Hadama proved the stability of steady states in a wide class, which includes Fermi gas at zero temperature in dimension greater than $3$, with strong condition for the potential function.

\subsection{Derivation of the equation}

Before introducing our result, we explain formally the derivation of the equation, inspired by the work of de Suzzoni in \cite{SLSEDP}. The dynamics of fermionic systems was also studied in \cite{Bardos2003Jun,Benedikter2014Nov}. We used the same method than de Suzzoni to obtain our model with a three body interaction: the particles are interacting through the potential $w$, a distribution of $(\R^d)^2$ of the form: $$w(u,v)=\Tilde{w}(u,v)+\Tilde{w}(v,u+v)+\Tilde{w}(u+v,u),$$ where $\Tilde{w}$ is a distribution of $(\R^d)^2$ which is even for each variables and symmetric. That is to say we considered a system of $N$ particles represented by the function $\psi:(\R^d)^{N}\to \C$ with the energy:
$$E_N=-\frac{1}{2}\sum_{i=1}^N\int_{(\R^d)^{N}}dx\ \Bar{\psi}\Delta_i \psi +\frac{1}{2}\sum_{i\neq j,i\neq k,j\neq k}\int_{(\R^d)^{N}}dx\ w(x_i-x_j,x_j-x_k)\Bar{\psi}\psi.$$

The first term of the energy is a term of kinetic energy of the system and the second one is the interaction part. 
 
We take $\psi$ of the form of a Slater determinant: $$\psi(x_1,...,x_N)=\frac{1}{\sqrt{N!}}\sum_{\sigma}\varepsilon(\sigma)\prod_iu_{\sigma(i)}(x_i),$$ 

where $(u_i)$ is an orthonormal family of $L^2(\R^d)$ and $\varepsilon(\sigma)$ is the signature of the permutation $\sigma$. In this case the system satisfies the Pauli principle for fermions.

Under some mean field assumptions, we get the standard system of $N$ coupled Schrödinger equations on $\R\times\R^d$, where $d$ is the space dimension:

\begin{equation}\label{partfinies}
  \forall j,\ i\partial_tu_j=-\Delta u_j+\bigg( \sum_{k,l=1,...,N}\int w(x_1-\cdot,x_2-\cdot)\lvert u_k\rvert^2(x_1)\lvert u_l\rvert^2(x_2)dx_1,dx_2\bigg)u_j.  
\end{equation}

If we take $X:\R\times\R^d\times\Omega\to \C$ a time dependent random field over $\R^d$ , where $(\Omega,\mathcal{A},d\omega)$ is the underlying probability space, of the form $X:=N^{-\frac{1}{2}}\sum u_ig_i$ where $(g_i)_{1\leq i\leq N}$ is an orthonormal family of $L^2(\Omega)$. Then $X$ is a solution of (\ref{eqprinc}). In the rest of the paper, we do not consider only random fields of the previous form $X:=N^{-\frac{1}{2}}\sum u_ig_i$, solutions of (\ref{eqprinc}), but time dependent random fields $X\in \mathcal{C}(\R,L^2(\Omega,H^{\frac{d-1}{2}}(\R^d))$.

\subsection{Main result}

We prove the stability of a non-localised equilibrium, through the proof of a scatterig result. Before we give the main result, we define the equilibrium we consider, given by Wiener integrals: 

\begin{equation}\label{Equilibrium}
    Y_f(t,x,\omega):=\int_{\xi\in\R^d}f(\xi)e^{i\xi x-it(m+\lvert\xi\rvert^2)}dW(\xi),
\end{equation}

for a function $f:\R^d\to \C$ and where we set: $$m=\Tilde{m}^2\int w(u,v)dudv=\bigg(\int_{\R^d}\lvert f \rvert^2\bigg)^2\int w(u,v)dudv.$$ Above, $dW(\xi)$ denotes infinitesimal complex Gaussian characterised by: $$\E(\overline{dW(\eta)}dW(\xi))=\delta(\eta-\xi)d\eta d\xi,$$ where $\delta$ is the Dirac delta function. The function $Y_f$ is a solution of (\ref{eqprinc}) (see section 3), but it is not stationary. However, its law is invariant under time and space translations, and in particular $Y_f$ is not localised. 

In this present work we prove the stability of the equilibria (\ref{Equilibrium}) for (\ref{eqprinc}), via the proof of scattering of perturbations to linear waves in their vicinity, as in \cite{CodS18}.

We set the following conditions for a spherically symmetric function $f$ in $L^2\cap L^\infty(\R^d)$, and $s_c=\frac{d-1}{2}$: 
\begin{description}
    \item[($A_1$)] $\langle \xi\rangle^{\lceil s_c\rceil}f\in L^2(\R^d)$,
    \item[($A_2$)] $\int_{\R^d}\lvert \xi\rvert^{1-d}\lvert f\nabla f\rvert <\infty$,
    \item[($A_3$)] writing $r=\lvert \xi\rvert $ the radial variable, $\partial_r\lvert f\rvert^2$ is continuous on $(0,\infty)$, $\partial_r\lvert f\rvert^2<0$ for $r>0$, and $r^\frac{d-1}{2}\partial_r\lvert f\rvert^2(r)\in L^1(0,\infty)$,
    \item[($A_4$)]writing $h$ the inverse Fourier transform of $ \lvert f\rvert^2$, $\langle x\rangle^2\partial^\alpha h\in L^\infty(\R^d)$ for all $\alpha\in \N^d$ with $\lvert \alpha\rvert\leq 3 \lceil s_c \rceil$,
    \item[($A_5$)] $\lvert x\rvert^{2-d}h\in L^1(\R^d)$.
\end{description}

Finally, we denote: $\underline{w}(u):=\int w(y,u)dy$.

\begin{theorem}\label{Thprinc}(Stability of equilibria for Equation (\ref{eqprinc})). Let $d\geq 4$ and $s_c=\frac{d-1}{2}$. Let f be a spherically symmetric function in $L^2\cap L^\infty(\R^d)$. Assume that $f$ satisfies assumptions $(A_1)$ to $(A_6)$.

Then there exists $C(f)$ such that for all finite measure $w$ of $(\R^d)^2$ even for each variables and symmetric, that satisfies: $\lvert\lvert (\hat{\underline{w}})_-\rvert\rvert_{L^\infty}\leq C(f)$ ; there exists $\varepsilon=\varepsilon(w,f)$ such that for any $Z_0\in H^{s_c},L_\omega^2\cap L_x^{\frac{d}{d+2}},L_\omega^2$ with: $$\lvert\lvert Z_0\rvert\rvert_{L_\omega^2H^{s_c}\cap L_x^{\frac{d}{d+2}}L_\omega^2}\leq \varepsilon ,$$

there is a solution of (\ref{eqprinc}) with initial data $X_0=Y_f(t=0)+Z_0$ which is global in both time directions. Moreover, it scatters to a linear solution in the sense that there exists $Z_-,Z_+\in L_\omega^2H^s$ such that: 
$$X(t)=Y_f(t)+e^{it(\Delta-m)}Z_\pm+o_{L_\omega^2H^s}(1)\ at \ t\to \pm \infty.$$

\end{theorem}

\begin{remark}
    The conditions on $f$ are satisfied by thermodynamical equilibria for bosonic or fermionic gazes at a positive temperature $T$:
    $$\lvert f(\xi)\rvert^2=\frac{1}{e^{\frac{\lvert\xi\rvert^2-\mu}{T}}-1},\ \ \mu<0,$$and: $$\lvert f(\xi)\rvert^2=\frac{1}{e^{\frac{\lvert\xi\rvert^2-\mu}{T}}+1},\ \ \mu\in\R,$$
    respectively. 

    Note that we have more assumptions on the function $f$ compared to the result for a two body interaction given by Collot and de Suzzoni, see Theorem 1 in \cite{CodS18}. But those assumptions are natural, they are due to the relaxation of the hypothesis on $w$ compared to \cite{CodS18}. However they are similar to the ones in \cite{CodS20} where the hypothesis on $w$ are similar to what we have.
\end{remark}

\begin{remark}
    Notice that the assumptions on the smallness of $\lvert\lvert (\hat{\underline{w}})_-\rvert\rvert_{L^\infty}$ already appears in the seminal work \cite{lewsabII} and were used also in \cite{chen2018scattering,CodS18}, but in these aforementioned work there is also an hypothesis on $\hat{\underline{w}}(0)_+$. Those hypothesis are needed to invert the linear part (see Section 5), but due to new arguments we do not need the hypothesis of smallness of $\hat{\underline{w}}(0)_+$ (the defocusing part of the equation).

\end{remark}

\begin{remark}

    In this paper we take $w$ a finite measure which implies that $\underline{w}$ is also (at most) a finite measure, this is to be compared with \cite{lewsabII} (Corollary 1), in which the authors use a potential in $L^1$ and to \cite{CodS20}. 
    
    We can also note that we take a potential with no regularity, on the contrary in \cite{CodS18} the potential is in a $H^{s,1}$ space, in order to treat the nonlinear part of their equation. However, in \cite{CodS20}, Collot and de Suzzoni aready gives method to take a finite measure as a potential. This latest result is in dimension $2$ and $3$, and we generalize the method to higher dimensions.  
    
\end{remark}

\begin{remark}
    We highlight that we work with the critical regularity for the quintic Schrödinger equation: $s_c=\frac{d-1}{2}$.
\end{remark}

\subsection{Ideas of proof}

The methods used in this paper were in part inspired by \cite{Naka}, by noticing similarities between the study of the stability of equilibria (\ref{Equilibrium}) for (\ref{eqprinc}) and the stability of the trivial solutions for the Gross-Pitaevskii equation $i\partial_t\psi=-\Delta\psi+(\lvert \psi\rvert^2-1)\psi$. In the two problems, the linearised dynamics has distinct dispersive properties at low and high frequencies, and different regularities for low and high frequencies need to be introduced, this is why in Section 2 we introduce inhomogeneous Besov spaces. We also refer to the work of Gustafson, Nakanishi and Tsai \cite{NakaGusTai2,NakaGus3,NakaGusTai1} for more information on scattering result for the Gross-Pitaevskii equation. 

The first intuition in \cite{CodS18} was for the authors to consider a Banach fixed point problem on a pair of variables: the perturbation $Z:=X-Y_f$ and another variable $V=\E[\lvert Z\rvert^2]+2Re\E(\Bar{Y}_f Z)$ which corresponds to the perturbation of the one particle density fonction. The authors then noticed that, for their framework, even if $Re\E(\Bar{Y}_f Z)$ is linear in $Z$, it has good properties regarding dispersion. It has similar integrability properties than $\E[\lvert Z\rvert^2]$ that is quadratic in $Z$, which leads to satisfying dispersive properties, sufficient to obtain the scattering result. However, this is not sufficient for our framework because we have less regularity on our potential compared to \cite{CodS18}.

In this paper, we need to exploit regularity in high frequencies on $V$. We split the variable $V$ into  two terms, a linear one with respect to $Z$: $V_1:=2Re\E(\Bar{Y}_f Z)$ and a quadratic one: $V_2:=\E[\lvert Z\rvert^2]$. We notice that, in low frequencies, $V_1$ and $V_2$ have indeed similar integrability. However, $V_1$ and $V_2$ have different integrability at high frequencies. In particular, $V_2$ displays better decay at infinity for high frequencies. We use this better decay to optimize the hypothesis on the initial condition. Performing a contraction argument on $V=V_1+V_2$ would not allow to observe this better decay at infinity of $V_2$. And the lesser decay of $V_1$ would allow to perform a contraction argument but only if the initial condition is very regular (strictly more than the critical regularity). This is why we divide $V$ into $V_1$ and $V_2$ and perform a contraction argument on $(Z,V_1,V_2)$, see Remark \ref{rmkCI}.

\subsection{Organization of the paper}

The paper is organised as follows. In Section 2, we give some notations and known results useful for the rest of the paper. In Section 3, we state the fixed point problem we need to solve in order to prove the theorem. In Section 4, we show some direct embeddings and Strichartz estimates. In Section 5, we study the linear terms, using previous works. In section 6, we identify a condition sufficient to treat all of the nonlinear terms, and we verifiy that this condition is verified in our framework. Estimates for the initial perturbation are showed in Section 7. Theorem \ref{Thprinc} is proven in section 8.  

\subsection{Acknowledgements}

The author would like to thank Prof. Anne-Sophie de Suzzoni and Prof. Charles Collot for their guidance and support. They both gave many most helpful comments during the writing of this paper. The author also thanks Prof. Daniel Han Kwan for his contribution to Proposition \ref{linearpropinversion1}. 
The author is supported by the ANR ESSED ANR-18-CE40-0028.

\section{Notations and preliminary results}\label{Notation}

\subsection{Besov spaces}

Here we take $\eta$ a smooth function with support included in the annulus $\{ \xi \in\R^d, \ \lvert\xi\rvert\in(1/2,2)\} $, we define for $j\in \Z$, $\eta_j(\xi):=\eta(2^{-j}\xi)$ and we assume that on $\R^d\backslash \{0\},\sum_j \eta_j=1 $. 

\begin{notation}
For any $u\in\mathcal{S}'$, we write $u_j=\Delta_j u$, where $\Delta_j$ is the Fourier multiplier by $\eta_j$. 

We have $u=\sum u_j$ and we call this the Littlewood-Paley decomposition of $u$.
\end{notation}

In the following definition, we define the inhomogeneous Besov spaces where the low and high frequencies are treated differently, in the following sense: 

\begin{definition}[Inhomogeneous Besov spaces]
Let $s,t\in \R$ and $p\geq 1$, we define $B_p^{s,t}$ by the space induced by the norm: $$\lvert\lvert u\rvert\rvert_{B_p^{s,t}} =\bigg(\sum_{j<0}2^{2js}\lvert\lvert u_j\rvert\rvert_{L^p}^2+\sum_{j\geq 0}2^{2jt}\lvert\lvert u_j\rvert\rvert_{L^p}^2\bigg)^{1/2}.$$
\end{definition}

\begin{remark}
    $B_p^{s,t}$ is a Banach space when $s\leq \frac{d}{p}$.
\end{remark}

\subsection{Notations}

\begin{notation}
We denote the japanese bracket by $\langle \xi\rangle:=(1+\lvert \xi\rvert^2)^{1/2}$.

We define by $S(t)$ the operator $S(t)=e^{-it(m-\Delta)}$. 

We use the standard notations for the Lebesgue and Bessel spaces.
\end{notation}

\begin{notation}(Fourier transform)
We define the Fourier transform with the following normalisation: $$\hat{g}(\xi)=\mathcal{F}(g)(\xi)=\int_{\R^d}g(x)e^{-ix\xi}dx,$$
and the inverse Fourier transform: $$\mathcal{F}^{-1}(g)(x)=(2\pi)^{-d}\int_{\R^d}g(\xi)e^{ix\xi}d\xi.$$
\end{notation}

\begin{notation}(Time-space norms) For $p,q\in [1,\infty]$ we denote $L_t^p,L_x^q$ the space $L^p(\R,L^q(\R^d))$, and for $s\in\R$ we denote $L_t^p,H_x^{s,q}$ the space $L^p(\R,H^{s,q}(\R^d)).$ In the case $s=2$ we also write $L_t^p,H_x^s=L_t^p,H_x^{s,2}.$ If $p=q$ we write $L_{t,x}^p=L_t^p,L_x^p$.

For $p,q\in[1,\infty],\ s,t\in\R$ we denote by $L^p_t,B_q^{s,t}$ the space $L^p(\R,B_q^{s,t}(\R^d))$.

\end{notation}

\begin{notation}(Probability-time-space norms) 
For $p,q\in [1,\infty]$ we denote $L_\omega^2,L_t^p,L_x^q$ the space $L^2(\Omega,L^p(\R,L^q(\R^d)))$, and for $s\in\R$ we denote $L_\omega^2,L_t^p,H_x^{s,q}$ the space $L^2(\Omega,L^p(\R,H^{s,q}(\R^d))).$ In the case $s=2$ we also write $L_\omega^2,L_t^p,H_x^s=L_\omega^2,L_t^p,H_x^{s,2}.$ 

For $p,q\in[1,\infty],\ s,t\in\R$ we denote by $L_\omega^2,L^p_t,B_q^{s,t}$ the space $L(\Omega,L^p(\R,B_q^{s,t}(\R^d)))$.
\end{notation}

\begin{notation}(Time-space-probability norm)
For $p,q\in [1,\infty]$ we denote $L_t^p,L_x^q,L_\omega^2$ the space $L^p(\R,L^q(\R^d,L^2(\Omega)))$, and for $s\in\R$ we denote $L_t^p,H_x^{s,q},L_\omega^2$ the space $$(1-\Delta_x)^{-\frac{s}{2}}L^p(\R,L^q(\R^d,L^2(\Omega))),$$
with the norm: $$\lvert\lvert u\rvert\rvert_{L_t^p,H_x^{s,q},L_\omega^2}=\lvert \lvert \langle \nabla\rangle^{s}u\rvert\rvert_{L_t^p,L_x^{q},L_\omega^2}.$$

In the case $s=2$ we also write $L_t^p,H_x^s, L_\omega^2=L_t^p,H_x^{s,2},L_\omega^2.$ 

For $p,q\in[1,\infty],\ s,t\in\R$ we denote by $L^p_t,B_q^{s,t},L_\omega^2=L^p,B_q^{s,t},L_\omega^2$ the space induced by the norm: $$\lvert\lvert u\rvert\rvert_{L^p,B_q^{s,t},L_\omega^2}=\lvert\lvert \bigg(\sum_{j<0}2^{2js}\lvert\lvert u_j\rvert\rvert^2_{L_x^q,L_\omega^2}+\sum_{j\geq 0}2^{2jt}\lvert\lvert u_j\rvert\rvert^2_{L_x^q,L_\omega^2}\bigg)^{\frac{1}{2}}\rvert\rvert_{L^p(\R)}.$$
\end{notation}

\subsection{Strichartz estimates}

First we give the usual Strichartz estimates, for more information on these estimates, we refer to \cite{GV4,GV3,GV,livreTao}. We recall that we work in dimension $d\geq 4$.

\begin{proposition} \label{strgen}
Let $(q,r)\in [2,\infty]^2$ such that: $\frac{2}{q}+\frac{d}{r}=\frac{d}{2}$. There exists $C$ such that for all $u_0\in L^2(\R^d)$: $$ \lvert\lvert S(t)u_0 \rvert \rvert_{L_t^qL_x^r}\leq C \lvert\lvert u_0 \rvert \rvert_{L^2}.$$
\end{proposition}

\begin{definition}

Any $(q,r)\in [2,\infty]^2$ such that: $\frac{2}{q}+\frac{d}{r}=\frac{d}{2}$ is said to be admissible.

\end{definition}

\begin{proposition}
Let $(q_1,r_1)$ and $(q_2,r_2)$ admissible. 

There exists $C$ such that for all $F\in L_t^{q_1'}L_x^{r_1'}(\R\times \R^d)$, and for all $u_0\in L^2(\R^d)$: 
\begin{equation}\label{str}
     \lvert\lvert S(t)u_0-i\int_0^t S(t-\tau)F(\tau)d\tau \rvert \rvert_{L_t^{q_2}L_x^{r_2}}\leq C \bigg[\lvert\lvert u_0 \rvert \rvert_{L^2} + \lvert\lvert F \rvert \rvert_{L_t^{q_1'}L_x^{r_1'}}\bigg].
\end{equation}

\end{proposition}

\begin{corollary}
Let $(q_1,r_1)$ and $(q_2,r_2)$ admissible. 

There exists $C$ such that for all $F\in L_t^{q_1'}H_x^{s,r_1'}$, and for all $u_0\in H^s$: 

\begin{equation}\label{strw}
\lvert\lvert S(t)u_0-i\int_0^t S(t-\tau)F(\tau)d\tau \rvert \rvert_{L_t^{q_2}H_x^{s,r_2}}\leq C \bigg[\lvert\lvert u_0 \rvert \rvert_{H^s} + \lvert\lvert F \rvert \rvert_{L_t^{q_1'}H_x^{s,r_1'}}\bigg].
\end{equation}

Moreover, let $(p,q,s)$ such that $r\geq 2$ where $\frac{1}{r}=\frac{1}{q}+\frac{s}{d}$ and $\frac{2}{p}+\frac{d}{q}+s=\frac{d}{2}$. There exists $C>0$ such that for all $u_0\in H^s$ and $F\in L_t^{q_1'}H_x^{s,r_1'}$: 
\begin{equation}\label{strhs}
    \lvert\lvert S(t)u_0-i\int_0^t S(t-\tau)F(\tau)d\tau\rvert \rvert_{L_t^{q}L_x^{p}}\leq C \bigg[\lvert\lvert u_0 \rvert \rvert_{H_x^s} + \lvert\lvert F \rvert \rvert_{L_t^{q_1'}H_x^{s,r_1'}}\bigg].
\end{equation}

Lastly we have the following dual Strichartz estimates. 

There exists $C$ such that for all $F\in L_t^{q_1'}H_x^{s,r_1'}$: 
\begin{equation}\label{strdual}
    \lvert\lvert \int_0^\infty S(-\tau)F(\tau)d\tau\rvert\rvert_{H^s}\leq C \lvert\lvert F \rvert\rvert_{L_t^{q_1'}H_x^{s,r_1'}}.
\end{equation}
\end{corollary}

The following proposition was introduced in \cite{CodS18} (Proposition 5.2) and gives motivation to the introduction of the inhomogeneous spaces. Indeed, we can see that a singularity "appears" in the low frequencies when we are dealing with a part of the linear flow (see Section 3). Using hypothesis $(A_4)$ we have that for $j\in \N, 0\leq j\leq \lceil s_c \rceil $ writing $h_j$ the inverse Fourier transform of $(\lvert \xi\rvert^j \lvert f\rvert)^2$, $\langle x\rangle^2\partial^\alpha h_j\in L^\infty(\R^d)$ for all $\alpha\in \N^d$ with $\lvert \alpha\rvert\leq 2 \lceil s_c \rceil$, then we can extend the result of Collot and de Suzzoni to the following result: 

\begin{proposition}\label{prop52}
Let $\sigma,\sigma_1,p_1,q_1$ be such that $\sigma,\sigma_1\geq 0,\sigma\leq \lceil s_c\rceil,p_1>2$ and: $$\frac{2}{p_1}+\frac{d}{q_1}=\frac{d}{2}-\sigma_1.$$

Assume moreover that for all $\alpha \in \N^d$ with $\lvert \alpha\rvert\leq 2 \lceil \sigma\rceil$, and for all $j\in \N,\ s\leq \lceil s_c\rceil$, $\lvert \partial^\alpha h_j\rvert\lesssim \langle x\rangle^{-2}$. Then there exists $C>0$ such that for all $U\in L_t^2,B_2^{-\frac{1}{2},\sigma_1+\sigma}$, and for all $0\leq j\leq \lceil s_c\rceil$:
\begin{equation}\label{prop521}
    \lvert\lvert \int_0^\infty S(t-s)\nabla^j Y_f(s)U(s)ds\rvert\rvert_{L_t^{p_1},B_{q_1}^{0,\sigma},L_\omega^2}\leq C \lvert\lvert U\rvert\rvert_{L_t^2,B_2^{-\frac{1}{2},\sigma_1+\sigma-\frac{1}{2}}}.\end{equation}

And if $\sigma+\sigma_1\leq \lceil s_c\rceil$, there exists $C>0$ such that for all $U\in L_t^2,B_2^{\sigma-\frac{1}{2},\sigma_1+\sigma}$:

\begin{equation}\label{prop522}\lvert\lvert \int_0^\infty S(t-s)\nabla^j Y_f(s)U(s)ds\rvert\rvert_{L_t^{p_1},H^{\sigma,q_1},L_\omega^2}\leq C \lvert\lvert U\rvert\rvert_{L_t^2,B_2^{\sigma_1-\frac{1}{2},\sigma_1+\sigma-\frac{1}{2}}}.\end{equation}
\end{proposition}

\begin{remark}
    In the paper \cite{CodS18}, the inequalities concern only $Y_f$ and not $\nabla^j Y_f$. But with the hypothesis $(A_4)$, as:
    $$\nabla^j Y_f=\int_{\R^d} \lvert \xi \rvert^j f(\xi)e^{ixt-it(\lvert \xi\rvert^2-m)}dW(\xi),$$
    we can apply the result of Collot and de Suzzoni to $\nabla^j Y_f$.
\end{remark}

\section{Strategy of the proof}\label{Strat}

We want to investigate the stability of the solution $Y_f$ of (\ref{eqprinc}). In the following we use the notation $Y$ instead of $Y_f$ to lighten the notations.

We begin by showing that $Y$ is indeed a solution. 

\begin{proposition}
    Under the hypothesis of Theorem \ref{Thprinc}, $Y$ is a solution of equation $(\ref{eqprinc})$.
\end{proposition}

\begin{proof} We verify that: 
$$Y(t)=e^{it\Delta}Y_0-i\int_0^t e^{i(t-\tau)\Delta}[w\circledast(\E[\lvert Y\rvert^2],\E[\lvert Y\rvert^2])Y(\tau)d\tau].$$

First we have that: 
$$\E[\lvert Y\rvert^2]=\int \lvert f\rvert^2=\Tilde{m},$$
and we get: $$w\circledast(\E[\lvert Y\rvert^2],\E[\lvert Y\rvert^2])=\Tilde{m}^2\int w(x_1-x,x_2-x)dx_1dx_2=m.$$

We then compute that:
$$e^{it\Delta}Y_0=e^{itm}Y(t),$$

and: $$e^{i(t-\tau)\Delta}Y(\tau)=e^{i(t-\tau)m}Y(t).$$

Combining the equalities above, we get: 
$$e^{it\Delta}Y_0-i\int_0^t e^{i(t-\tau)\Delta}[w\circledast(\E[\lvert Y\rvert^2],\E[\lvert Y\rvert^2])Y(\tau)d\tau]=Y(t)[e^{itm}-im\int_0^t e^{i(t-\tau)m}d\tau]=Y(t).$$
\end{proof}

We consider a perturbation of the equilibrium $X=Z+Y$ that is a solution of (\ref{eqprinc}). We have formally:

$$i\partial_t(Y+Z)=-\Delta(Y+Z)+\int dx_1dx_2 w(x_1-\cdot,x_2-\cdot)\E[\lvert Y+Z\rvert^2](x_1)\E[\lvert Y+Z\rvert^2](x_2)(Y+Z).$$

Using that $\E[\lvert Y\rvert^2]=\Tilde{m}$ and introducing $V:=2Re\E[\Bar{Y}Z]+\E[\lvert Z\rvert^2]$, we obtain the following system: 

$$ \left \{ \begin{array}{l}
     i\partial_t (Y+Z)Z=-\Delta Z+\int dx_1dx_2 w(x_1-\cdot,x_2-\cdot)(\Tilde{m}+V(x_1))(\Tilde{m}+V(x_2))(Y+Z)-\Delta Y
  \\
     V=\E[\lvert Z\rvert^2]+2Re\E[\Bar{Y}Z].
\end{array} \right . $$

Using that $Y$ is a solution of (\ref{eqprinc}) and because $w$ is invariant by permutation we obtain:  

$$ \left \{ \begin{array}{l}
     i\partial_t Z=(m-\Delta)Z+\int dx_1dx_2 w(x_1-\cdot,x_2-\cdot)(2\Tilde{m}V(x_2)+V(x_1)V(x_2))(Y+Z)
  \\
     V=\E[\lvert Z\rvert^2]+2Re\E[\Bar{Y}Z].
\end{array} \right . $$

Then we use that $\int dx_1 w(x_1-x_3,x_2-x_3)=\underline{w}(x_2-x_3)$ and we get:

$$\left \{ \begin{array}{l}
     i\partial_t Z=(m-\Delta)Z+2\Tilde{m}(\underline{w}*V)(Y+Z)+(w\circledast(V,V))(Y+Z)
  \\
     V=\E[\lvert Z\rvert^2]+2Re\E[\Bar{Y}Z],
\end{array} \right . $$
recalling that: \begin{equation}
   w\circledast(V,V)=\int  w(x_1-x_2,x_2-\cdot)V(x_1)V(x_2)dx_1dx_2.
\end{equation}

Thus, giving an initial perturbation $Z_0$, we want to solve the following Cauchy problem:

\begin{equation}\label{eqperturb}
\left \{ \begin{array}{l}
     i\partial_t Z=(m-\Delta)Z+2\Tilde{m}(\underline{w}*V)(Y+Z)+(w\circledast(V,V))(Y+Z)
  \\
     V=\E[\lvert Z\rvert^2]+2Re\E[\Bar{Y}Z]\\
     Z_{| t=0}=Z_0.
\end{array} \right . 
\end{equation}

We use Duhamel's formula on $Z$ and we inject it on the last term of $V$, thus we get that:

\begin{equation}
\left \{ \begin{array}{l}
     Z=S(t)Z_0+2\Tilde{m}W_V(Y)+2\Tilde{m}W_V(Z)+\Tilde{W}_{V,V}(Y)+\Tilde{W}_{V,V}(Z)
  \\
     V=\E[\lvert Z\rvert^2]+2Re\E[\Bar{Y}(S(t)Z_0+2\Tilde{m}W_V(Y)+2\Tilde{m}W_V(Z)+\Tilde{W}_{V,V}(Y)+\Tilde{W}_{V,V}(Z))],
\end{array} \right . 
\end{equation}

with the respectively bilinear and trilinear operators: \begin{equation}
    \begin{array}{l}
         W_V(Z)(t):=-i\int_0^t S(t-s)(\underline{w}*V(s))Z(s)ds,\\
         \Tilde{W}_{V,V}(Z)(t):=-i\int_0^t S(t-s)(w\circledast(V,V)(s))Z(s)ds. 
    \end{array}
\end{equation}

We split the variable $V$ as two other variables $V_1$ and $V_2$:

\begin{equation*}
\left \{ \begin{array}{l}
     Z=S(t)Z_0+2\Tilde{m}(W_{V_1+V_2}(Y+Z)+\Tilde{W}_{V_1+V_2,V_1+V_2}(Z+Y)
  \\
     V_1=2Re\E[\Bar{Y}(S(t)Z_0+2\Tilde{m}(W_{V_1+V_2}(Y+Z)+\Tilde{W}_{V_1+V_2,V_1+V_2}(Z+Y)]\\
     V_2=\E[\lvert Z\rvert^2].
\end{array} \right . 
\end{equation*}

Then $(Z,V_1,V_2)^T$ verifies: 

\begin{equation}\label{ptfixe}
    \begin{pmatrix}
    Z\\ 
    V_1\\
    V_2
    \end{pmatrix}=C_{Z_0}+L\begin{pmatrix}
    Z\\ 
    V_1\\
    V_2
    \end{pmatrix} +R\begin{pmatrix}
    Z\\ 
    V_1\\
    V_2
    \end{pmatrix},
\end{equation}

where the constant term is: \begin{equation}\label{constantpart}
    C_{Z_0}=\begin{pmatrix}
    S(t)Z_0\\
    2Re\E[\Bar{Y}S(t)Z_0]\\
    0
    \end{pmatrix},
\end{equation}

the linearised operator is:

\begin{equation}\label{linearpart}
    L\begin{pmatrix}
    Z\\ 
    V_1\\
    V_2
    \end{pmatrix}=\begin{pmatrix}
    2\Tilde{m}W_{V_1}(Y)+2\Tilde{m}W_{V_2}(Y)\\
    4\Tilde{m}Re\E[\Bar{Y}W_{V_1}(Y)]+4\Tilde{m}Re\E[\Bar{Y}W_{V_2}(Y)]\\
    0
    \end{pmatrix},
\end{equation}

and the nonlinear remainder term is:

\begin{equation}\label{nonlin}
    R\begin{pmatrix}
    Z\\ 
    V_1\\
    V_2
    \end{pmatrix}=\begin{pmatrix}
    2\Tilde{m}W_{V_1+V_2}(Z)\Tilde{W}_{V_1+V_2,V_1+V_2}(Z+Y)\\
    4\Tilde{m}Re\E[\Bar{Y}W_{V_1+V_2}(Z)]+2(Re\E[\Bar{Y}\Tilde{W}_{V_1+V_2,V_1+V_2}(Y)]+Re\E[\Bar{Y}\Tilde{W}_{V_1+V_2,V_1+V_2}(Z)])\\
    \E[\lvert Z\rvert^2]
    \end{pmatrix}
\end{equation}

To solve the Cauchy problem (\ref{eqperturb}) we have to solve the fixed point problem (\ref{ptfixe}).

To do this we need to find the right function spaces to consider. We set the following spaces:

\begin{equation}
    \begin{array}{rcr}
         \Theta_Z&=&\mathcal{C}(\R,H^{s_c}(\R^d,L_\omega^2))\cap L^{p_*}(\R,H^{s_c,p_*}(\R^d,L^2_\omega))\cap L^{2(d+2)}(\R,L^{2(d+2)}(\R^d,L_\omega^2)) \\
         & & \cap L^4(\R,H_x^{s_c,\frac{2d}{d-1}}(\R^d,L_\omega^2)), 
         
    \end{array}
\end{equation}

where $s_c=\frac{d-1}{2}$ and $p_*=\frac{2(d+2)}{d}$.

And let: 

$$\Theta_V=\Theta_{V_1}+\Theta_{V_2}=L_t^2,B_2^{-\frac{1}{2},s_c-\frac{1}{2}}\cap\Big( L_{t,x}^{2(d+2)}\cap L_t^{p_*},H_x^{s_c,p_*}+L_{t,x}^{d+2}\cap L_t^{2\frac{d+2}{d+1}},H_x^{s_c,2\frac{d+2}{d+1}}\Big).$$

\section{Bilinear inequality and Strichartz estimates}

In this section we present linear and bilinear inequalities on the spaces and operators we introduced in section 3.

\subsection{A Bilinear inequality}

Before we present a bilinear inequality, we give an interpolated Leibniz rule lemma, useful in all of the article, which is a consequence of a result proved by Gulisashvili and Kon \cite{gulisashvili1996exact}.

\begin{lemma}\label{interpolationargument}(\cite{gulisashvili1996exact} Theorem 1.4)
    Let $s>0$ and $r,p_1,q_1,p_2,q_2$ such that for $i=1,2$:$$\frac{1}{r}=\frac{1}{p_i}+\frac{1}{q_i}.$$

    Then for $u\in L^{p_1}\cap H^{s,p_2}$ and $v\in H^{s,q_1}\cap L^{q_2}$, we have that $uv\in H^{s,r}$ with: 

    $$\lvert\lvert uv\rvert\rvert_{H^{s,r}}\lesssim \lvert\lvert u\rvert\rvert_{L^{p_1}\cap H^{s,p_2}}\lvert\lvert v\rvert\rvert_{L^{q_2}\cap H^{s,q_1}}.$$
\end{lemma}

We now give the embedding result.

\begin{proposition}\label{propembedding}
We have for all $u,v\in\Theta_Z$: \begin{equation}\label{embedding}
    \lvert\lvert \E(uv) \rvert\rvert_{\Theta_{V_2}}\lesssim \lvert\lvert u \rvert\rvert_{\Theta_Z} \lvert\lvert v \rvert\rvert_{\Theta_Z},
\end{equation}
in other terms the application $\gamma:(u,v)\mapsto uv$ is continous from $\Theta_Z\times \Theta_Z$ into $\Theta_{V_2}$
\end{proposition}

\begin{proof}

\textbf{Step 0:} We have that $\gamma$ is continuous from $L_{t,x}^{2(d+2)}\times L_{t,x}^{2(d+2)}$ into $L_{t,x}^{d+2}$ by the Hölder inequality.

\textbf{Step 1:} We now show that $\gamma$ is continuous from $\Theta_Z\times\Theta_Z $ into $L_t^2B_2^{-\frac{1}{2},s_c-\frac{1}{2}}$. For this we treat the high and low frequencies separately. Let $(u,v)\in \Theta_Z\times\Theta_Z$.  

We have by definition of the Besov spaces: \begin{equation}\label{embedBes}
\lvert\lvert \E[uv]\rvert\rvert_{L_t^2,B_2^{-\frac{1}{2},s_c-\frac{1}{2}}}\leq \lvert\lvert\bigg(\sum_{j<0}2^{-j}\lvert\lvert \E[uv]_j\rvert\rvert^2_{L^2}\bigg)^{\frac{1}{2}}\rvert\rvert_{L_t^2}+\lvert\lvert\bigg(\sum_{j\geq 0}2^{2j(s_c-\frac{1}{2})}\lvert\lvert \E[uv]_j\rvert\rvert^2_{L^2}\bigg)^{\frac{1}{2}}\rvert\rvert_{L_t^2} .\end{equation}

We first look at the first term, the low frequencies. We have:

\begin{equation*}
    \lvert\lvert\bigg(\sum_{j<0}2^{-j}\lvert\lvert \E[uv]_j\rvert\rvert^2_{L^2}\bigg)^{\frac{1}{2}}\rvert\rvert_{L_t^2}\leq \lvert\lvert \E[uv]\rvert\rvert_{L_t^2,B_2^{-\frac{1}{2},-\frac{1}{2}}}\lesssim \lvert\lvert \E[uv]\rvert\rvert_{L_t^2,\Dot{H}^{-\frac{1}{2}}}. 
\end{equation*}

We then use Sobolev embeddings to get: 

$$\lvert\lvert\bigg(\sum_{j<0}2^{-j}\lvert\lvert \E[uv]_j\rvert\rvert^2_{L^2}\bigg)^{\frac{1}{2}}\rvert\rvert_{L_t^2}\lesssim \lvert\lvert \E[uv]\rvert\rvert_{L_t^2,L_x^{\frac{2d}{d+1}}}.$$

Finally, by the Hölder inequality: 

\begin{equation}\label{embedlf}
    \lvert\lvert\bigg(\sum_{j<0}2^{-j}\lvert\lvert \E[uv]_j\rvert\rvert^2_{L^2}\bigg)^{\frac{1}{2}}\rvert\rvert_{L_t^2}\lesssim \lvert\lvert u\rvert\rvert_{L_t^4,L_x^{\frac{4d}{d+1}}}\lvert\lvert v\rvert\rvert_{L_t^4,L_x^{\frac{4d}{d+1}}}\leq \lvert\lvert u\rvert\rvert_{\Theta_Z}\lvert\lvert v\rvert\rvert_{\Theta_Z},
\end{equation}

because $H^{s_c,\frac{2d}{d-1}}$ is embedded in $L^{\frac{4d}{d+1}}$.

We now look at high frequencies. We first write that by the Littlewood-Paley theorem: 

$$\lvert\lvert\bigg(\sum_{j\geq 0}2^{2j(s_c-\frac{1}{2})}\lvert\lvert \E[uv]_j\rvert\rvert^2_{L^2}\bigg)^{\frac{1}{2}}\rvert\rvert_{L_t^2}\leq \lvert\lvert \E[uv]\rvert\rvert_{L_t^2,B_2^{0,s_c-\frac{1}{2}}}\lesssim \lvert\lvert \E[uv]\rvert\rvert_{L_t^2,H^{s_c-\frac{1}{2}}}.$$

So we need to estimate $\lvert\lvert \langle \nabla \rangle^{s_c-\frac{1}{2}} \E[uv]\rvert\rvert_{L_{t,x}^2}$. 

We use Lemma \ref{interpolationargument} and the Hölder inequality to get: 

$$\lvert\lvert \langle \nabla \rangle^{s_c-\frac{1}{2}} \E[uv]\rvert\rvert_{L_{t,x}^2}\lesssim \lvert\lvert u\rvert\rvert_{L_t^4,L_x^{2d},L_\omega^2\cap L_t^4,H_x^{s_c-\frac{1}{2},\frac{2d}{d-1}},L_\omega^2}\lvert\lvert v\rvert\rvert_{L_t^4,L_x^{2d},L_\omega^2\cap L_t^4,H_x^{s_c-\frac{1}{2},\frac{2d}{d-1}},L_\omega^2},$$

because: $\frac{1}{2}=\frac{1}{2d}+\frac{d-1}{2d}$.

Thus, we get: 
\begin{equation}\label{embedhf}
    \lvert\lvert \langle\nabla\rangle^{s_c-\frac{1}{2}}\E[uv] \rvert\rvert_{L_{t,x}^2}\lesssim \lvert\lvert u\rvert\rvert_{\Theta_Z}\lvert\lvert v\rvert\rvert_{\Theta_Z}.
\end{equation}

Combining (\ref{embedBes}), (\ref{embedlf}) and (\ref{embedhf}) gives:

$$\lvert\lvert \E[uv]\rvert\rvert_{L_t^2,B_2^{-\frac{1}{2},s_c-\frac{1}{2}}}\lesssim \lvert\lvert u\rvert\rvert_{\Theta_Z}\lvert\lvert v\rvert\rvert_{\Theta_Z}.$$

\textbf{Step 3:} We are left to prove that $\gamma$ is continuous from $\Theta_Z\times\Theta_Z$ into $L_t^{\frac{2(d+2)}{d+1}},H_x^{s_c,{\frac{2(d+2)}{d+1}}}$.

We use the Lemma \ref{interpolationargument} we have: 

$$\lvert\lvert \E[uv] \rvert\rvert_{L_{t}^{\frac{2(d+2)}{d+1}},H_x^{s_c,\frac{2(d+2)}{d+1}}}\lesssim \lvert\lvert u\rvert\rvert_{L_t^{p_*},H_x^{s_c,p_*},L_\omega^2\cap L_{t,x}^{2(d+2)},L_\omega^2} \lvert\lvert v\rvert\rvert_{L_t^{p_*},H_x^{s_c,p_*},L_\omega^2\cap L_{t,x}^{2(d+2)},L_\omega^2}.$$ 

So, we conclude that: $$\lvert\lvert \E[uv]\rvert\rvert_{L_{t}^{\frac{2(d+2)}{d+1}},H_x^{s_c,\frac{2(d+2)}{d+1}}}\lesssim \lvert\lvert u\rvert\rvert_{\Theta_Z}\lvert\lvert v\rvert\rvert_{\Theta_Z}.$$

\end{proof}

\subsection{Strichartz estimates}

In this subsection we give some Strichartz type estimates, that we use in section 6 to treat the nonlinear terms in $(\ref{nonlin})$. 

We notice that all of the nonlinearities contained in (\ref{nonlin}) are of the form: $\int_0^t S(t-s)W(s)ds$ or: $\E[\Bar{Y}\int_0^t S(t-s)W(s)ds]$.  We aim at giving a sufficient condition on $W$ to estimate these nonlinearities. This is the subject of Propositions \ref{NLZ} and \ref{NLV1}. The proofs of these two propositions relies on Propositions \ref{Strprop} and \ref{NLV1lemma}.

\begin{proposition}\label{Strprop}
There exists $C$ such that for all $u\in L_\omega ^2,H^{s_c}$, \begin{equation}\label{Str}
    \lvert \lvert S(t)u\rvert\rvert_{\Theta_Z}\leq C\lvert \lvert u\rvert\rvert_{L_\omega ^2,H^{s_c}} .
\end{equation}
\end{proposition}

\begin{proof}
By the Minkowski inequality for all $u\in L_\omega^2,\big(\mathcal{C}(\R,H^{s_c})\cap L^{p_*},H^{s_c,p^*}\cap L^{2(d+2)},L^{2(d+2)}\cap L_t^4,L_x^{4\frac{d}{d+1}}\cap L^4(\R,H^{s_c,\frac{2d}{d-1}})\cap L^4(\R,L^{2d})\big)$, $$\lvert \lvert u\rvert\rvert_{\Theta_Z}\leq \lvert \lvert u\rvert\rvert_{L_\omega^2,\big(\mathcal{C}(\R,H^{s_c})\cap L^{p_*},H^{s_c,p^*}\cap L^{2(d+2)},L^{2(d+2)}\cap L^4(\R,H^{s_c,\frac{2d}{d-1}})\big)}.$$

For the $L^{p_*},H^{s_c,p_*}$, $L^4,H^{s_c,\frac{2d}{d-1}}$ and $\mathcal{C}(\R,H^{s_c})$ estimates we use the Strichartz estimates (\ref{strw}) with $F=0$, as $(p_*,p_*)$, $(4,\frac{2d}{d-1})$ and $(\infty,2)$ are admissible. 

For the $L^{2(d+2)},L^{2(d+2)}$ estimate we notice that because of (\ref{strhs}) $S(t)$ for any $p,q\in[2,\infty]$ such that: $\frac{2}{p}+\frac{d}{q}\in [\frac{1}{2},\frac{d}{2}]$, $S(t)u\in L_\omega^2,L_t^p,L_x^q$ and: 

$$\lvert\lvert S(t)u\rvert\rvert_{L_\omega^2,L_t^p,L_x^q}\lesssim \lvert\lvert u\rvert\rvert_{L_\omega^2,H^s}\leq \lvert\lvert u\rvert\rvert_{L_\omega^2,H^{s_c}}. $$

We just compute that for $p=q=2(d+2)$: 

$$\frac{2}{2(d+2)}+\frac{d}{2(d+2)}=\frac{1}{2},$$

and we can conclude that: 

$$\lvert\lvert S(t)u\rvert\rvert_{L_\omega^2,L_t^{2(d+2)},L_x^{2(d+2)}}\lesssim\lvert\lvert u\rvert\rvert_{L_\omega^2,H^{s_c}}. $$

\end{proof}

\begin{proposition}\label{NLZ}
    Let $W\in L_t^{p_*'},H_x^{s_c,p_*'},L_\omega^2$ then $\int_0^t S(t-s)W(s)ds\in \Theta_Z$ and: $$\lvert\lvert \int_0^t S(t-s)W(s)ds\lvert\lvert_{\Theta_Z}\lesssim \lvert\lvert W\rvert\rvert_{L_t^{p_*'},H_x^{s_c,p_*'},L_\omega^2}.$$
\end{proposition}

\begin{proof}
    By Proposition \ref{Strprop} we have: 
$$\lvert\lvert \int_0^\infty S(t)S(-\tau)W(\tau) d\tau\rvert\rvert_{\Theta_Z}\leq C\lvert\lvert \int_0^\infty S(-\tau)W(\tau) d\tau\rvert\rvert_{L^2_\omega H^{s_c}}. $$

Then we use the Strichartz Dual estimates (\ref{strdual}) to get: $$\lvert\lvert \int_0^\infty S(t)S(-\tau)W(\tau)d\tau\rvert\rvert_{\Theta_Z}\leq C\lvert\lvert W\rvert\rvert_{L_\omega^2,L_t^{p_*'},H_x^{s_c,p_*'}}.$$

Thus by using the Christ-Kiselev Lemma \cite{CKLemma} we get: $$\lvert\lvert \int_0^t S(t-\tau)W(\tau)d\tau\rvert\rvert_{\Theta_Z}\leq C\lvert\lvert W \rvert\rvert_{L_\omega^2,L_t^{p_*'},H_x^{s,{p_*'}}}.$$

Finally, as $p_*'<2$, by the Minkowski inequality: 
$$\lvert\lvert \int_0^t S(t-\tau)W(\tau)d\tau\rvert\rvert_{\Theta_Z}\leq C\lvert\lvert W \rvert\rvert_{L_t^{p_*'},H_x^{s_c,p_*'},L_\omega^2}.$$
\end{proof}

To prove Proposition \ref{NLV1lemma}, we need the following lemma, used in \cite{CodS20} (Proposition 4.7) in dimension $3$, this is a consequence of Proposition \ref{prop52}.

\begin{lemma}\label{lemmestrY2}
    For any $2\leq p,q\leq \infty$ and $0\leq s<\frac{d}{2}$ such that $\frac{2}{p}+\frac{d}{q}=\frac{d}{2}-s$ there exists $C>0$ such that for all $U\in L_t^2\Dot{H}_x^s$ and for $,j\in \N,\ 0\leq j\leq \lceil s_c\rceil$: \begin{equation}\label{strdY2}
        \lvert\lvert\int_{\R} S(t-\tau)(\nabla^j Y(\tau)U(\tau))d\tau\rvert\rvert_{L_t^pL_x^qL_\omega^2}\lesssim\lvert\lvert U\rvert\rvert_{L_t^2\Dot{H}_x^{s-\frac{1}{2}}},
    \end{equation} 
    with a constant depending on $f$.
\end{lemma}

\begin{proposition}\label{NLV1lemma}
    Let $1\leq s\leq \lceil s_c\rceil $ and let $W\in\ L_t^{p_*'},H_x^{s-1,p_*'},L_\omega^2$, then $\E[\Bar{Y}\int_0^t S(t-\tau)W(\tau)d\tau]\in L_t^2,\Dot{H}_x^{s-\frac{1}{2}}$, and:       $$\lvert\lvert\E[\Bar{Y}\int_0^t S(t-\tau)W(\tau)d\tau]\rvert\rvert_{L_t^2,\Dot{H}_x^{s-\frac{1}{2}}} \lesssim \lvert \lvert W\rvert\rvert_{L_t^{p_*'},H_x^{s-1,p_*'},L_\omega^2}.$$
\end{proposition}
\begin{proof}
    We suppose that $s\in \N$  and we conclude by interpolation. 

    We use duality to show that $\nabla^{s-\frac{1}{2}}\E[\Bar{Y}\int_0^t S(t-\tau)W(\tau)d\tau]\in L_{t,x}^2$. Let $U\in L_{t,x}^2$, we have by Leibniz formula: 

    $$\langle U,\nabla^{s-\frac{1}{2}}\E[\Bar{Y}\int_0^t S(t-\tau)W(\tau)d\tau]\rangle=\langle \nabla^{\frac{1}{2}}U, \E[\sum_{j_1+j_2=s-1}\nabla^{j_1}\Bar{Y}\int_0^t S(t-\tau)\nabla^{j_2}W(\tau)d\tau ]\rangle.$$

    Then: $$\langle U,\nabla^{s-\frac{1}{2}}\E[\Bar{Y}\int_0^t S(t-\tau)W(\tau)d\tau]\rangle=\sum_{j_1+j_2=s-1}\E[\langle \int_t^{\infty}S(t-\tau)\nabla^{j_1}Y(\tau)\nabla^{\frac{1}{2}}U(\tau)d\tau, \nabla^{j_2}W \rangle].$$

    Then we claim that by Lemma \ref{lemmestrY2}, with $p=q=p_*$ we have that for any $j_1$: $$\lvert\lvert \int_t^{\infty}S(t-\tau)\nabla^{j_1}Y(\tau)\nabla^{\frac{1}{2}}U(\tau)d\tau\rvert\rvert_{L_{t,x}^{p_*},L_\omega^2}\lesssim \lvert\lvert \nabla^\frac{1}{2}U\rvert\rvert_{\Dot{H}^{-\frac{1}{2}}_{t,x}}=\lvert\lvert U\rvert\rvert_{L_{t,x}^2}.$$

    And by the Hölder inequality: 

    $$\lvert \E[\langle \int_t^{\infty}S(t-\tau)\nabla^{j_1}Y(\tau)\nabla^{\frac{1}{2}}U(\tau)d\tau, \nabla^{j_2}W \rangle]\rvert \leq \lvert\lvert \int_t^{\infty}S(t-\tau)\nabla^{j_1}Y(\tau)\nabla^{\frac{1}{2}}U(\tau)d\tau\rvert\rvert_{L_{t,x}^{p_*},L_\omega^2} \lvert\lvert \nabla^{j_2} W\rvert\rvert_{L_{t,x}^{p_*'},L_\omega^2}.$$

Combining the two inequalities above gives: $$\lvert \E[\langle \int_t^{\infty}S(t-\tau)\nabla^{j_1}Y(\tau)\nabla^{\frac{1}{2}}U(\tau)d\tau, \nabla^{j_2}W \rangle\rvert\lesssim  \lvert\lvert U\rvert\rvert_{L_{t,x}^2}\lvert\lvert \nabla^{j_2} W\rvert\rvert_{L_{t,x}^{p_*'},L_\omega^2}$$

    So we get: \begin{equation*}
        \lvert \langle U,\nabla^{s-\frac{1}{2}}\E[\Bar{Y}\int_0^t S(t-\tau)W(\tau)d\tau]\rangle\rvert \lesssim \lvert\lvert U\rvert\rvert_{L_{t,x}^2} \lvert \lvert W\rvert\rvert_{L_t^{p_*'},H_x^{s-1,p_*'},L_\omega^2}.
    \end{equation*}

    Then by duality: \begin{equation}\label{LemmmaNLV1ineg1}
        \lvert\lvert\nabla^{s_c-\frac{1}{2}}\E[\Bar{Y}\int_0^t S(t-\tau)W(\tau)d\tau]\rvert\rvert_{L_{t,x}^2} \lesssim \lvert \lvert W\rvert\rvert_{L_t^{p_*'},H_x^{s-1,p_*'},L_\omega^2},
    \end{equation}
and the proof is complete.
\end{proof}

\begin{proposition}\label{NLV1}
    Let $W\in L_t^{p_*'},H_x^{s_c,p_*'},L_\omega^2\cap L_{t,x}^{\frac{2(d+2)}{d+6}},L_\omega^2$ then $\lvert\lvert\E[\Bar{Y}\int_0^t S(t-\tau)W(\tau)d\tau]\in \Theta_{V_1}$ and: $$\E[\Bar{Y}\int_0^t S(t-\tau)W(\tau)d\tau]\lvert\lvert_{\Theta_{V_1}}\lesssim \lvert\lvert W\rvert\rvert_{L_t^{p_*'},H_x^{s_c,p_*'},L_\omega^2}+\lvert\lvert W\rvert\rvert_{L_{t,x}^{\frac{2(d+2)}{d+6}},L_\omega^2}.$$
\end{proposition}

\begin{proof}
    First we notice that $\langle \nabla\rangle^j Y\in L_{t,x}^\infty,L_\omega^2$ for any $0\leq j\leq \lceil s_c\rceil$ by hypothesis on $f$.

    Thus by the Hölder inequality and definition of $\Theta_Z$: $$\lvert\lvert \E[\Bar{Y}\int_0^t S(t-\tau)W(\tau)d\tau]\rvert\rvert_{L_{t,x}^{2(d+2)}\cap L_t^{p_*},H_x^{s_c,p_*}}\lesssim \lvert\lvert \int_0^t S(t-\tau)W(\tau)d\tau\rvert\rvert_{\Theta_Z}. $$

    We conclude by using Proposition \ref{NLZ} that: 
    $$\lvert\lvert \E[\Bar{Y}\int_0^t S(t-\tau)W(\tau)d\tau]\rvert\rvert_{L_{t,x}^{2(d+2)}\cap L_t^{p_*},H_x^{s_c,p_*}}\lesssim \lvert\lvert W\rvert\rvert_{L_t^{p_*'},H_x^{s_c,p_*'},L_\omega^2}.$$

    We are left to show the following estimate: 

    \begin{equation}\label{NLV1BES}
        \lvert\lvert \E[\Bar{Y}\int_0^t S(t-\tau)W(\tau)d\tau]\rvert\rvert_{L_t^2,B_2^{-\frac{1}{2},s_c-\frac{1}{2}}}\lesssim \lvert\lvert W\rvert\rvert_{L_t^{p_*'},H_x^{s_c,p_*'},L_\omega^2}+\lvert\lvert W\rvert\rvert_{L_{t,x}^{\frac{2(d+2)}{d+6}},L_\omega^2}.
    \end{equation}

    To do this we are going to look to both high and low frequencies, using: \begin{equation}\label{highlowtktk}
    \begin{array}{rcl}
         \lvert\lvert \E[\Bar{Y}\int_0^t S(t-\tau)W(\tau)d\tau]\rvert\rvert_{L_t^2,B_2^{-\frac{1}{2},s_c-\frac{1}{2}}}& \leq  & \lvert\lvert \sum_{j<0} \bigg(\E[\Bar{Y}\int_0^t S(t-\tau)W(\tau)d\tau]\bigg)_j\rvert\rvert_{L_t^2,B_2^{-\frac{1}{2},s_c-\frac{1}{2}}} \\
         & & +\lvert\lvert \sum_{j\geq 0} \bigg(\E[\Bar{Y}\int_0^t S(t-\tau)W(\tau)d\tau]\bigg)_j\rvert\rvert_{L_t^2,B_2^{-\frac{1}{2},s_c-\frac{1}{2}}}.
    \end{array}
    \end{equation}

    \textbf{Low frequencies:} in this subpart, we want to control the $L_t^2,B_2^{-\frac{1}{2},s_c-\frac{1}{2}}$ norm by duality, but only for the low frequency term in (\ref{highlowtktk}). So we take $U\in L_t^2,B_2^{\frac{1}{2},\frac{1}{2}-s_c}$of the form: $$U=\sum_{j\leq 0} U_j.$$

    We write: 

    $$\langle U, \E[\Bar{Y}\int_0^t S(t-\tau)W(\tau)d\tau]\rangle =\E\langle \int_t^\infty S(t-\tau)Y(\tau)U(\tau)d\tau,W\rangle.$$

    We use the Hölder inequality and inequality (\ref{prop522}) with $(p,q,\sigma_1,\sigma)=(\frac{2(d+2)}{d-2},\frac{2(d+2)}{d-2},1,0)$ to get: 

$$\begin{array}{rl}
     \lvert\langle U, \E[\Bar{Y}\int_0^t S(t-\tau)W(\tau)d\tau]]\rangle \rvert & \leq \lvert\lvert \int_t^\infty S(t-\tau)Y(\tau)U(\tau)d\tau\rvert\rvert_{L_{t,x}^{\frac{2(d+2)}{d-2}},L_\omega^2}\lvert\lvert W\rvert\rvert_{L_{t,x}^{\frac{2(d+2)}{d+6}},L_\omega^2} \\
     & \lesssim \lvert\lvert U\rvert\rvert_{B_2^{1-\frac{1}{2},1-\frac{1}{2}}}\lvert\lvert W\rvert\rvert_{L_{t,x}^{\frac{2(d+2)}{d+6}},L_\omega^2},
\end{array}$$

    Because: $\frac{2(d-2)}{2(d+2)}+\frac{d(d-2)}{2(d+2)}=\frac{d}{2}-1$ and: $\frac{d-2}{2(d+2)}+\frac{d+6}{2(d+2)}=1$.

And by considering that $U$ contains only low frequency we obtain: $$\lvert\langle U, \E[\Bar{Y}\int_0^t S(t-\tau)W(\tau)d\tau]]\rangle \rvert\lesssim \lvert\lvert U\rvert\rvert_{B_2^{\frac{1}{2},\frac{1}{2}-s_c}}\lvert\lvert W\rvert\rvert_{L_{t,x}^{\frac{2(d+2)}{d+6}},L_\omega^2}.$$

Thus by duality: 

\begin{equation}\label{NLV1LowFreq}
    \lvert\lvert \sum_{j<0} \bigg(\E[\Bar{Y}\int_0^t S(t-\tau)W(\tau)d\tau]\bigg)_j\rvert\rvert_{L_t^2,B_2^{-\frac{1}{2},0}}\lesssim \lvert\lvert W\rvert\rvert_{L_{t,x}^{\frac{2(d+2)}{d+6}},L_\omega^2}.
\end{equation}

\textbf{High frequencies:} to control the high frequency term in (\ref{highlowtktk}), we use Proposition \ref{NLV1lemma} to get:
$$\lvert\lvert\E[\Bar{Y}\int_0^t S(t-\tau)W(\tau)d\tau]\rvert\rvert_{L_t^2,\Dot{H}_x^{s_c-\frac{1}{2}}} \lesssim \lvert \lvert W\rvert\rvert_{L_t^{p_*'},H_x^{s_c-1,p_*'},L_\omega^2}.$$

And: $$\lvert\lvert \sum_{j\geq 0} \bigg(\E[\Bar{Y}\int_0^t S(t-\tau)W(\tau)d\tau]\bigg)_j\rvert\rvert_{L_t^2,B_2^{0,s_c-\frac{1}{2}}}\leq \lvert\lvert\E[\Bar{Y}\int_0^t S(t-\tau)W(\tau)d\tau]\rvert\rvert_{L_t^2,\Dot{H}_x^{s_c-\frac{1}{2}}}.$$

Thus: \begin{equation}\label{NLV1HighFreq}
    \lvert\lvert \sum_{j\geq 0} \bigg(\E[\Bar{Y}\int_0^t S(t-\tau)W(\tau)d\tau]\bigg)_j\rvert\rvert_{L_t^2,B_2^{0,s_c-\frac{1}{2}}}\lesssim \lvert \lvert W\rvert\rvert_{L_t^{p_*'},H_x^{s_c,p_*'},L_\omega^2}.
\end{equation}

Combining (\ref{NLV1LowFreq}) and (\ref{NLV1HighFreq}) gives: 

$$\lvert\lvert \E[\Bar{Y}\int_0^t S(t-\tau)W(\tau)d\tau]\rvert\rvert_{L_t^2,B_2^{-\frac{1}{2},s_c-\frac{1}{2}}}\lesssim \lvert \lvert W\rvert\rvert_{L_t^{p_*'},H_x^{s_c,p_*'},L_\omega^2}+\lvert\lvert W\rvert\rvert_{L_{t,x}^{\frac{2(d+2)}{d+6}},L_\omega^2}.$$

And the proof is complete.

\end{proof}

\section{Linear term}

In this section, we study the linear operator $L$ defined in (\ref{linearpart}) and we prove Proposition \ref{proplinpartinv}. The objective is to invert $1-L$. 

We denote $L_1:V\mapsto -4\Tilde{m}Re\E[\Bar{Y}W_V(Y)] $ and $L_2:V\mapsto -2\Tilde{m}W_V(Y)$.

\begin{proposition}\label{proplin}
The operator $$L_2 : V \mapsto -W_V(Y)$$
is continuous from $L_t^2,B_2^{-\frac{1}{2},s_c-\frac{1}{2}}$ into $\Theta_Z$, so it is continuous from $\Theta_{V}$ to $\Theta_Z$.
\end{proposition}

\begin{proof}

Let $V \in L_t^2,B_2^{-\frac{1}{2},s_c-\frac{1}{2}}$. By using Proposition \ref{prop52} and the Christ-Kiselev Lemma, for any $(\sigma,\sigma_1,p_1,q_1)$ such that: $$\sigma\geq 0,\ \sigma_1\geq 0,\ \sigma+\sigma_1 \leq \lceil s \rceil ,\ p_1>2,$$ and: $$\frac{2}{p_1}+\frac{d}{q_1}=\frac{d}{2}-\sigma_1,$$ we have:

\begin{equation}\label{prop52b}
    \lvert\lvert \int_0^t S(t-s)Y(s)U(s)ds\lvert\lvert_{L_t^{p_1},H^{\sigma,q_1},L_\omega^2}\leq C \lvert \lvert U\rvert\rvert_{L_t^2,B_2^{\sigma_1-\frac{1}{2},\sigma_1+\sigma-\frac{1}{2}}}.
\end{equation}

By applying (\ref{prop52b}) to $(p_1,q_1,\sigma_1,\sigma)$ equal to either: $(p_*,p_*,0,s_c),\ (2(d+2),2(d+2),s_c,0),\ $ $\ (4,\frac{2d}{d-1},s_c,0) ,\text{ or }(\infty,2,0,s_c),$
we get: $$\lvert\lvert W_V(Y)\rvert\rvert_{\Theta_Z}\leq C \lvert\lvert \underline{w}*V\rvert\rvert_{L_t^2,B_2^{-\frac{1}{2},s_c-\frac{1}{2}}}.$$

And thus by Young's inequality: $$\lvert\lvert W_V(Y)\rvert\rvert_{\Theta_Z}\leq C  \rvert\rvert V\rvert\rvert_{L_t^2,B_2^{-\frac{1}{2},s_c-\frac{1}{2}}}.$$

The proof is complete.
\end{proof}

\begin{lemma}
The operator $L_1$ is a time-space Fourier multiplier of symbol $4\Tilde{m}\hat{\underline{w}}m_f$ where $m_f$ is the Lindhard function associated to $f$, defined by the formula: \begin{equation}
    m_f(\tau,\xi)=2\mathcal{F}_t\big(sin(\lvert \xi\rvert^2t)h(2\xi t)\mathbf{1}_{t\geq 0}\big)(\tau),
\end{equation}
where we recall that $h$ is the inverse Fourier transform of $\lvert f\rvert^2$.
\end{lemma}

\begin{proof}
It has been proven in \cite{CodS18}, where the authors use the properties of Wiener integrals to get the result.
\end{proof}

\begin{proposition}\label{linearpropinversion1} (Lewin-Sabin \cite{lewsabII}, see also Collot-de Suzzoni \cite{CodS18})
Let $d\geq 4$ and $s=\frac{d}{2}-1$. Let $f$ be a radial map in $L^2\cap L^\infty(\R^d)$. Assume that: 
\begin{itemize}
    \item $\langle \xi \rangle^{\lceil s_c\rceil} f\in L^2(\R^d)$
    \item $\int_{\R^d} \lvert\xi\rvert^{1-d}\lvert f\nabla f\rvert <\infty$,
    \item writing $r=\lvert \xi\rvert $ the radial variable, $\partial_r\lvert f\rvert^2$ is continuous on $(0,\infty)$, $\partial_r\lvert f\rvert^2<0$ for $r>0$, and $r^\frac{d-1}{2}\partial_r\lvert f\rvert^2(r)\in L^1$,
    \item $\lvert \xi\rvert^{2-d}h\in L^1(\R^d)$.
\end{itemize}

If $\underline{w}$ is an even finite measure such that: $$\lvert\lvert (\hat{\underline{w}})_-\rvert\rvert_{L^\infty}\bigg(\int_{\R^d}\frac{\lvert h\rvert}{\lvert x \rvert^{d-2}}dx\bigg)<2\lvert \mathbb{S}^{d-1}\rvert,$$

then there holds \begin{equation}\label{inversionmultFourier}
    \underset{(\omega,\xi)\in \R\times\R^d}{inf} \lvert \hat{\underline{w}}(k) m_f(\omega,k)+1\rvert>0, 
\end{equation}

and $1+L_1$ is continous, invertible and of continuous inverse on $L^2_{t,x}$ and $L_t^2B_2^{-\frac{1}{2},s_c-\frac{1}{2}}$.
\end{proposition}

\begin{proof}
    See \cite{lewsabII} where the authors show an analogous result but with more hypotheses. The following new arguments follows from a private conversation with Daniel Han Kwan. 
    
    In \cite{lewsabII}, the authors show the condition (\ref{inversionmultFourier}) and claim that it suffices to invert the operator. However, in order to invert the operator we need actually a Penrose stability condition, see for instance \cite{Mouhot2011Jan,Penrose1960Mar}. We explain below this new condition. 
    
    The Fourier multiplier $m_f$ we consider is defined by: $$m_f(\omega,k)=2\int_0^\infty e^{-it\omega}sin(\lvert\xi\rvert^2t)h(2\xi t)dt.$$  

    We introduce for $\gamma \geq 0$ the quantity $m_h^\gamma$: 

    $$m_h^\gamma(\omega,k)=2\int_0^\infty e^{-t(\gamma+i\omega)}sin(\lvert\xi\rvert^2t)h(2\xi t)dt.$$

    We will denote: $m_h=m_h^0=m_f$.
    
    In order to invert the operator, the condition needed is: 

    \begin{equation}\label{Penrosecondition}
        \underset{\gamma\geq 0, \ (\omega,k)\in\R\times\R^d}{inf}\lvert \underline{\hat{w}}(k)m_h^\gamma(\omega,k)+1\rvert>0.
    \end{equation}

    Still we begin by showing that the condition (\ref{inversionmultFourier}) holds, and implies the new condition (\ref{Penrosecondition}).

    \textbf{Step 1: Reduction to the vincinity of the origin} 
    
    We recall briefly the arguments from \cite{lewsabII}, that are still valid under our hypotheses, showing that for any $V$ neighborhood of $(0,0)$: 
    $$\underset{(\omega,k)\in \R\times\R^d\backslash V}{inf}\lvert 1+\hat{\underline{w}}(k)m_f(\omega,k)\rvert>0.$$ 
    
    Indeed, $m_f$ is continuous on $\R\times\R^d\backslash \{(0,0)\}$ and there is a compact set $A$ of $\R\times\R^d$ such that on the complement of this set: 
    $$\lvert \hat{\underline{w}}m_f+1\rvert\geq \frac{1}{2}.$$

    In this paper, we use that the potential is a finite measure and not necessarily in $L^1$. The hypothesis of a potential in $L^1$ is used in the proof of Lewin and Sabin \cite{lewsabII} to have that: $$\underline{\hat{w}}(\xi)m_f(\tau,\xi)\to 0,$$ when $\lvert \xi\rvert\to \infty$ uniformly on $\tau$. But $m_f$ verifies that
    $m_f(\tau,\xi)\to 0$ when $\lvert \xi\rvert\to \infty$ uniformly on $\tau$. Then we only need that $\underline{\hat{w}}\in L^\infty$, which is true when $\underline{w}$ is a finite measure.

    To show the continuity and most of the following results we use the Lindhard function in dimension $d$ for the particular case $\lvert f\rvert^2=\textbf{1}(\lvert k\rvert^2\leq \mu)$ denoted by $m_d^F$, then:

    \begin{equation}\label{Fermi1D} 
    \begin{array}{rl}
         m_1^F(\mu,\omega,k)= & \frac{\sqrt{2}}{\sqrt{\pi}\lvert k\rvert}\Bigg[\frac{1}{4}log\bigg{\lvert}\frac{(\lvert k\rvert^2+2\lvert k\rvert\mu)^2-\omega^2}{(\lvert k\rvert^2-2\lvert k\rvert\mu)^2-\omega^2}\bigg{\rvert}\\
         & +i\frac{\pi}{4}\bigg(\mathbf{1}(\lvert\omega+\lvert k\rvert^2\rvert\leq 2\sqrt{\mu}\lvert k\rvert)-\mathbf{1}(\lvert\omega-\lvert k\rvert^2\rvert\leq 2\sqrt{\mu}\lvert k\rvert)\bigg)\Bigg]  
    \end{array}
    \end{equation}

for $d=1$ and for $d\geq 2$:

\begin{equation}\label{FermiDD}
    m_d^F(\mu,\omega,k)=\frac{\lvert \mathbb{S}\rvert^{d-2}}{(2\pi)^d}\mu^{\frac{d-1}{2}}\int_0^1 m_1^F(\mu(1-r^2),\omega,k)r^{d-2}dr.
\end{equation}

    We note that: 
    \begin{equation}\label{mfLS}
        m_f(\omega,k)=-\int_{0}^\infty m_d^F(s,\omega,k)\partial_s \lvert f\rvert^2(s) ds.
    \end{equation}
    
    We first treat the space $\{\omega=0\}$ where by hypothesis: 
    $$\lvert m_f(0,k)\hat{\underline{w}}(k)+1\rvert\geq 1-\lvert\lvert \hat{\underline{w}}_-\rvert\rvert_{L^\infty}\frac{1}{2\lvert \S^{d-1}\rvert}\Bigg(\int_{\R^d}\frac{\lvert h(x)\rvert}{\lvert x\rvert^{d-2}}dx\Bigg)>0,$$
    and in $\{k=0\}$: $m_f(\omega,0)=0$.

    Then, the imaginary part of $m_f$ does not vanish outside of $\{k=0\}\cup \{\omega=0\}$, by using that for $\omega>0$ and $k\neq 0$, in dimension $1$ (for higher dimensions it is similar, we detail the computation for higher dimensions in step 2): \begin{equation}\label{imaginarypart}
        Im\Big(m_f(\omega,k)\Big)\leq \frac{\sqrt{\pi}}{2\sqrt{2}\lvert k \rvert}\int_{\frac{(\omega-\lvert k\rvert^2)^2}{4\lvert k\rvert^2}}^{\frac{(\omega+\lvert k\rvert^2)^2}{4\lvert k\rvert^2}}\partial_\tau\lvert f\rvert^2(\tau)d\tau <0,
    \end{equation}and that $m_f(\omega,k)=-m_f(-\omega,k)$. We can conclude that for any $V$ neighborhood of $(0,0)$: 
    $$\underset{(\omega,k)\in \R\times\R^d\backslash V}{inf}\lvert 1+\hat{\underline{w}}(k)m_f(\omega,k)\rvert>0.$$

    \textbf{Step 2: Study near the origin}
    
    We are left to study $\lvert \hat{w}m_f+1\rvert$ on a neighborhood of $(0,0)$ and we have to show: 
    
    $$\exists V\in \mathcal{V}_{\R\times\R^d}((0,0)), \exists \varepsilon_0>0, \forall (\omega,k)\in V, \lvert 1+\hat{\underline{w}}(k)m_f(\omega,k)\rvert\geq \varepsilon_0.$$
    
    We begin at the points where $Re(m_f)>0$. In this case we use the hypothesis on $(\hat{\underline{w}})_-$ and the fact that $m_f$ is uniformly bounded by:
    $\frac{1}{2\lvert \S\rvert^{d-1}}\int_{\R^d}\frac{\lvert h\rvert}{\lvert x \rvert^{d-2}}dx, $
    we get: $$1+\hat{\underline{w}}Re(m_f)\geq 1-(\hat{\underline{w}})_-Re(m_f)\geq \frac{1}{2}.$$ 
    
    For the other points, we use a remark in \cite{lewsabII}: 
    \begin{equation}\label{realpart}
        Re(m_f(\omega,k))\underset{k,\omega\to 0}{\sim} \frac{1}{2}\int_0^{\infty} th(t) cos\bigg( \frac{\omega}{2\lvert k \rvert}t\bigg)dt.
    \end{equation}
    
    But we have that: 

    $$\int_0^{\infty} th(t) cos( a t)dt\underset{a\to 0}{\to} \int_0^{\infty} th(t) dt>0,$$

    and we are in a case previously treated ($\Re(m_f(\omega,k))>0$). Moreover: $$\int_0^{\infty} th(t) cos( a t)dt\underset{a\to \infty}{\to} 0,$$

    thus when $\frac{\omega}{\lvert k\rvert}\to 0$ or when $\frac{\omega}{\lvert k\rvert}\to \infty$, the real part of $\hat{\underline{w}}(k)m_f(\omega,k)$ stays far from $-1$. 

    Then we only have to look at what happens when $\frac{\omega}{\lvert k\rvert}\in [c,R]$, for $c>0$ and $R<\infty$. In this case, without the assumptions introduced by Lewin and Sabin, we cannot be sure that the real part of $\Hat{\underline{w}}(k)m_f(\omega,k)$ stays far from $1$. 

    We have two cases, either $\Hat{\underline{w}}(0)=0$ and in for $k$ small enough $\hat{\underline{w}}(k)m_f(\omega,k)$ stays far from $-1$, either $\Hat{\underline{w}}(0)\neq 0$ and we show the imaginary part of $\Hat{\underline{w}}(k)m_f(\omega,k)$ is far from $0$. 

    We use (\ref{mfLS}) and (\ref{FermiDD}) to write that there is a constant $C_d>0$ such that: 

    $$Im\Big(m_f(\omega,k)\Big)=-C_d\int_0^\infty \int_0^1s^{\frac{d-1}{2}}\partial_s \lvert f\rvert^2(s) Im\Big( m_1^F(s(1-r^2),\omega,k)\Big)r^{d-2}drds.$$

    And because $Im\bigg(m_1^F(\cdot,\omega,k)\bigg)\in L^\infty$ and $r^\frac{d-1}{2}\partial_r \lvert f\rvert^2(r)\in L^1$, by Fubini's theorem: 

    $$Im\Big(m_f(\omega,k)\Big)=-C_d\int_0^1\int_0^\infty s^{\frac{d-1}{2}}\partial_s \lvert f\rvert^2(s) Im\Big( m_1^F(s(1-r^2),\omega,k)\Big)r^{d-2}dsdr.$$

    Using the case in dimension $1$ (inequality (\ref{imaginarypart})) we get, for $\omega>0$:

    \begin{eqnarray*}
        Im\Big(m_f(\omega,k)\Big) &\leq & C_d\frac{1}{\lvert k\rvert}\int_0^1 r^{d-2}\int_{\frac{(\omega-\lvert k\rvert^2)^2}{4\lvert k\rvert^2(1-r^2)}}^{\frac{(\omega+\lvert k\rvert^2)^2}{4\lvert k\rvert^2(1-r^2)}}\partial_s\lvert f\rvert^2(s)s^{\frac{d-1}{2}}ds dr\\
        & \leq & C_d\frac{1}{\lvert k\rvert}\int_0^\frac{1}{2} r^{d-2}\int_{\frac{(\omega-\lvert k\rvert^2)^2}{4\lvert k\rvert^2(1-r^2)}}^{\frac{(\omega+\lvert k\rvert^2)^2}{4\lvert k\rvert^2(1-r^2)}}\partial_s\lvert f\rvert^2(s)s^{\frac{d-1}{2}}ds dr,
    \end{eqnarray*}

    because $\partial_\tau\lvert f\rvert^2(\tau)<0$.

    Then we use the fact that $\frac{\omega}{\lvert k\rvert}\in [c,R]$, and the continuity of $\partial_r\lvert f\rvert^2$, to get for $\lvert k\rvert$ small enough (ie $\lvert k\rvert<c)$ that there exist a constant $C_f>0$ such that:

    \begin{eqnarray*}
        Im\Big(m_f(\omega,k)\Big) &\leq & -C_d\frac{1}{\lvert k\rvert}\int_0^\frac{1}{2} r^{d-2}\int_{\frac{(\omega-\lvert k\rvert^2)^2}{4\lvert k\rvert^2(1-r^2)}}^{\frac{(\omega+\lvert k\rvert^2)^2}{4\lvert k\rvert^2(1-r^2)}} C_f s^{\frac{d-1}{2}}ds dr\\
        &\leq & -C_d C_f \frac{1}{\lvert k\rvert}\int_0^\frac{1}{2} r^{d-2} \bigg( \frac{(c+\lvert k\rvert)^2}{4(1-r^2)}\bigg)^{\frac{d-1}{2}}\bigg( \frac{(\frac{\omega}{\lvert k\rvert}+\lvert k\rvert)^2}{4(1-r^2)}-\frac{({\lvert k\rvert}-\lvert k\rvert)^2}{4(1-r^2)}\bigg)dr.
    \end{eqnarray*}

Thus there is a constant $C_d'$ such that: 

    \begin{eqnarray*}
        Im\Big(m_f(\omega,k)\Big) &\leq & -C_d' C_f\frac{\omega}{\lvert k\rvert}\\
        &\leq & -C_d' C_f c .
    \end{eqnarray*}

Then, there exist a constant $C_1>0$ (depending on $f$) such that for $\lvert k\rvert$ small enough and for any $(\omega,k)$ such that $\frac{\omega}{\lvert k\rvert}\in [c,R]$ and $\omega>0$, we have: 

$$Im\Big(m_f(\omega,k)\Big)\leq -C_1<0.$$

If $\omega<0$ we use: $Im\Big(m_f(\omega,k)\Big)=-Im\Big(m_f(-\omega,k)\Big)$.

Thus, if we take a neighborhood of $(0,0)$ small enough, we can find a $\epsilon_0>0$ such that on this neighborhood: $$\lvert 1+\Hat{\underline{w}}m_F\rvert\geq \epsilon_0,$$

and this concludes the proof of (\ref{inversionmultFourier}) by combining step 1 and 2.

\textbf{Conclusion: }We now show that (\ref{Penrosecondition}) holds. 

Firstly, the quantity $m_h^\gamma (\omega,k)$ is continuous on $\gamma$, thus, as (\ref{inversionmultFourier}) holds, there exist $\lambda>0$ such that the condition becomes: 

$$\underset{\gamma\geq \lambda,\ (\omega,k)\in\R\times\R^d}{inf}\lvert \underline{\hat{w}}(k)m_h^\gamma(\omega,k)+1\rvert>0.$$

Then, as in the case $\gamma=0$, we consider the case where $\lvert f\rvert^2=\mathbf{1}(\lvert k\rvert^2\leq \mu)$ denoted by $m_d^{F,\gamma}$. And we have that: 

$$Re(m_1^{F,\gamma}(\mu,\omega,k))=\frac{\sqrt{2}}{4\sqrt{\pi}\lvert k\rvert} log\bigg(\frac{(\gamma^2+(\lvert k\rvert^2+2\sqrt{\mu}\lvert k\rvert-\omega)^2)(\gamma^2+(\lvert k\rvert^2+2\sqrt{\mu}\lvert k\rvert+\omega)^2)}{(\gamma^2+(\lvert k\rvert^2-2\sqrt{\mu}\lvert k\rvert-\omega)^2)(\gamma^2+(\lvert k\rvert^2+2\sqrt{\mu}\lvert k\rvert+\omega)^2)}\bigg).$$

Thus, using relations (\ref{FermiDD}) and (\ref{mfLS}) we have that: 

$$Re(m_h^\gamma (\omega,k))\underset{(\omega,k)\to 0}{\to}0,$$

uniformly in $\gamma\geq \lambda$. Then, there exists $V_0$ a neighborhood of $0$ in $\R\times \R^d$ such that the condition (\ref{Penrosecondition}) becomes: 

$$\underset{\gamma\geq \lambda,\ (\omega,k)\in\R\times\R^d\backslash V_0}{inf}\lvert \underline{\hat{w}}(k)m_h^\gamma(\omega,k)+1\rvert>0.$$

Then, using that: 

$$m_h^\gamma(\omega,k) \underset{\gamma \to \infty}{\to} 0,$$

uniformly in $(\omega,k)\in \R\times\R^d\backslash V_0$. Then we need to show that for a certain $\gamma_0>0$ and $\gamma_1>0$:

$$\underset{\gamma_0\leq \gamma \leq \gamma_1,\ (\omega,k)\in\R\times\R^d\backslash V_0}{inf}\lvert \underline{\hat{w}}(k)m_h^{\gamma}(\omega,k)+1\rvert>0.$$

Finally, we do a change of variable $\gamma'\lvert k\rvert=\gamma$ and we are going to prove that for some $\gamma_1'>0$: 

$$\underset{\gamma \leq \gamma_1',\ (\omega,k)\in\R\times\R^d\backslash V_0}{inf}\lvert \underline{\hat{w}}(k)m_h^{\gamma\lvert k\rvert}(\omega,k)+1\rvert>0.$$

We use the continuity of the function and show that for any $\gamma \leq \gamma_1',\ (\omega,k)\in\R\times\R^d\backslash V_0$ we have: $$\lvert \underline{\hat{w}}(k)m_h^{\gamma\lvert k\rvert}(\omega,k)+1\rvert>0.$$

We write that: $m_h^{\gamma\lvert k\rvert}=m_{\Tilde{h}}$, where $\Tilde{h}(x)=e^{-\gamma\lvert x\rvert/2}h(x)$. Then the $\Tilde{f}$ function associated such that $m_{\Tilde{h}}=m_{\Tilde{f}}$, is characterized by: 
$$\lvert \Tilde{f}\rvert^2=\lvert f\rvert^2 *\mathcal{F}(e^{\frac{-\gamma \lvert\cdot\rvert}{2}}).$$
So, $\lvert \Tilde{f}\rvert^2$ is the convolution of $\lvert f\rvert^2$ with an $L^1$ function. Thus $\Tilde{f}$ verifies the same condition as $f$ and we can apply (\ref{inversionmultFourier}) with $\Tilde{f}$ instead of $f$. 

Thus, for any $\gamma \leq \gamma_1',\ (\omega,k)\in\R\times\R^d\backslash V_0$ we have: $$\lvert \underline{\hat{w}}(k)m_h^{\gamma\lvert k\rvert}(\omega,k)+1\rvert>0.$$

And we conclude that (\ref{Penrosecondition}) holds.

\end{proof}

\begin{proposition}\label{proplineV1V2}
$L_1$ is continuous from $L_t^2,B_2^{-\frac{1}{2},s_c-\frac{1}{2}} $ into $\Theta_{V_1}$. 
\end{proposition}

\begin{proof}
By Proposition \ref{linearpropinversion1} we have that $L_1$ is continuous from $L_t^2,B_2^{-\frac{1}{2},s_c-\frac{1}{2}}$ into $L_t^2,B_2^{-\frac{1}{2},s_c-\frac{1}{2}}$.

Using the fact that $\langle\nabla\rangle^{s_c} Y\in L_{t,x}^\infty,L_\omega^2$, we have by Lemma \ref{interpolationargument}: 

\begin{equation}\label{lineV1V2}
    \lvert\lvert \E[\Bar{Y}L_2(V)]\rvert\rvert_{L_{t,x}^{2(d+2)}\cap L_t^4,H_x^{s_c,2\frac{d}{d+1}}\cap L_t^{p_*},H_x^{s_c,p_*}}\lesssim \lvert\lvert L_2(V)\rvert\rvert_{L_{t,x}^{2(d+2)},L_\omega^2\cap L_t^4,H_x^{s_c,2\frac{d}{d+1}},L_\omega^2\cap L_t^{p_*},H_x^{s_c,p_*},L_\omega^2}.
\end{equation}

Using (\ref{lineV1V2}) and Proposition \ref{proplin} we get that $L_1$ is continuous from $L_t^2,B_2^{-\frac{1}{2},s_c-\frac{1}{2}} $ into $L_{t,x}^{2(d+2)}\cap L_t^{p_*},H_x^{s_c,p_*}$. 

Then $L_1$ is continuous from $L_t^2,B_2^{-\frac{1}{2},s_c-\frac{1}{2}} $ into $\Theta_{V_1}$.

\end{proof}

\begin{proposition}\label{corollin}
    $L_1$ is continuous from $\Theta_{V_2}$ into $\Theta_{V_1}$ and $1+L_1$ is continuous, invertible of continuous inverse from $\Theta_{V_1}$ into $\Theta_{V_1}$. 
\end{proposition}

\begin{proof}
    By Proposition \ref{proplineV1V2} we get that $L_1$ is continuous from $\Theta_{V_2}$ into $\Theta_{V_1}$ because $\Theta_{V_2}$ is continuously embedded in $L_t^2,B_2^{-\frac{1}{2},s_c-\frac{1}{2}} $.

    Moreover, it follows from Proposition \ref{proplineV1V2} that $1+L_1$ is continuous from $\Theta_{V_1}$ into $\Theta_{V_1}$, because $\Theta_{V_1}$ is continuously embedded in $L_t^2,B_2^{-\frac{1}{2},s_c-\frac{1}{2}} $.

    To obtain that $1+L_1$ is invertible of continuous inverse from $\Theta_{V_1}$ into $\Theta_{V_1}$ we write that: $$(1+L_1)^{-1}=1-L_1(1+L_1)^{-1}.$$

    We use Proposition \ref{linearpropinversion1} that gives that $(1+L_1)^{-1}$ is continuous from $L_t^2,B_2^{-\frac{1}{2},s_c-\frac{1}{2}}$ into $L_t^2,B_2^{-\frac{1}{2},s_c-\frac{1}{2}}$ and then we use Proposition \ref{proplineV1V2} to get that $L_1(1+L_1)^{-1}$ is continuous from $L_t^2,B_2^{-\frac{1}{2},s_c-\frac{1}{2}} $ into $\Theta_{V_1}$. So $(1+L_1)^{-1}$ is continuous from $L_t^2,B_2^{-\frac{1}{2},s_c-\frac{1}{2}}$ into $\Theta_{V_1}$, therefore $(1+L_1)^{-1}$ is continuous from $\Theta_{V_1}$ into $\Theta_{V_1}$.
\end{proof}

\begin{proposition}\label{proplinpartinv}
    The operator $1-L$ is continuous, invertible with continuous inverse on $\Theta_Z\times \Theta_{V_1}\times \Theta_{V_2}$.
\end{proposition}

\begin{proof}
    We have that $1-L$ is of the form: \begin{equation}
    1-L=\begin{pmatrix}
    1 &  L_2 & L_2 \\ 0 & 1+L_1 & L_1\\ 0 & 0 & 1
    \end{pmatrix}.
\end{equation} It is sufficient to show that $1+L_1$ is continous, invertible with continuous inverse from $\Theta_{V_1}$ to $\Theta_{V_1}$, that $L_2$ is continuous from $\Theta_{V_1}$ to $\Theta_Z$ and from $\Theta_{V_2}$ to $\Theta_Z$, and that $L_1$ is continuous from $\Theta_{V_2}$ to $\Theta_{V_1}$. We use Propositions \ref{proplin} and \ref{corollin} to conclude. 
\end{proof}

\section{Nonlinear terms}

In this part we tackle all of the remaining nonlinear terms in (\ref{nonlin}). Formally, we identify two types of nonlinearity: quadratic and cubic. This part is subdivised in two subsections, in the first one, we give estimates on the terms corresponding to the quadratic nonlinearity and, in the second subpart, to the ones corresponding to the cubic nonlinearity.

\subsection{Quadratic terms}

We begin this subpart by Proposition \ref{wvzprop}, in which we give estimates on the terms $W_{V_1}(Z)$ and $W_{V_2}(Z)$. To treat the other quadratic terms we need to introduce Lemma \ref{lemmeconvoldast}, which gives an inequality about the product: $(u,v)\to w\circledast(u,v)$. In Proposition \ref{quad} we give inequality estimates on $\Tilde{W}_{V_1+V_2,V_1+V_2}(Y)$.

\begin{proposition}\label{wvzprop} For all $Z\in\Theta_Z$ and $V=V_1+V_2\in\Theta_{V}$,

\begin{equation}\label{wvz}
    \lvert\lvert W_V(Z)\rvert\rvert_{\Theta_Z}\lesssim \lvert\lvert Z\rvert\rvert_{\Theta_Z}\lvert\lvert V\rvert\rvert_{\Theta_{V}},
\end{equation}
and, \begin{equation}\label{wvzb}
    \lvert\lvert Re\E(\Bar{Y}W_V(Z))\rvert\rvert_{\Theta_{V_1}}\lesssim\lvert\lvert V\rvert\rvert_{\Theta_{V}}\lvert\lvert Z\rvert\rvert_{\Theta_Z}.   
\end{equation}
\end{proposition}

\begin{proof}
By Proposition \ref{NLZ}:
$$\lvert\lvert W_V(Z)\rvert\rvert_{\Theta_Z}\lesssim \lvert\lvert(\underline{w}*V)Z\rvert\rvert_{L^{p_*'},H^{s_c,p_*'},L_\omega^2},$$

and by Proposition \ref{NLV1} 

$$\lvert\lvert Re\E(\Bar{Y}W_V(Z))\rvert\rvert_{\Theta_{V_1}}\lesssim \lvert\lvert(\underline{w}*V)Z\rvert\rvert_{L^{p_*'},H^{s_c,p_*'},L_\omega^2}+\lvert\lvert (\underline{w}*V)Z\rvert\rvert_{L_{t,x}^{\frac{2(d+2)}{d+6}},L_\omega^2}.$$

It suffices to control $\lvert\lvert(\underline{w}*V)Z\rvert\rvert_{L^{p_*'},H^{s_c,p_*'},L_\omega^2}$ and $\lvert\lvert (\underline{w}*V)Z\rvert\rvert_{L_{t,x}^{\frac{2(d+2)}{d+6}},L_\omega^2}$ in order to conclude.

\textbf{First estimate:} We begin with $\lvert\lvert(\underline{w}*V)Z\rvert\rvert_{L^{p_*'},H^{s_c,p_*'},L_\omega^2}$, we have: $$\lvert\lvert(\underline{w}*V)Z\rvert\rvert_{L^{p_*'},H^{s_c,p_*'},L_\omega^2}\leq \lvert\lvert(\underline{w}*V_1)Z\rvert\rvert_{L^{p_*'},H^{s_c,p_*'},L_\omega^2}+\lvert\lvert(\underline{w}*V_2)Z\rvert\rvert_{L^{p_*'},H^{s_c,p_*'},L_\omega^2}.`$$ 
We use Lemma \ref{interpolationargument} and we get by adding Young's inequality:

$$\lvert\lvert(\underline{w}*V_1)Z\rvert\rvert_{L^{p_*'},H^{s_c,p_*'},L_\omega^2}\lesssim \lvert \lvert V_1\rvert\rvert_{L_t^{p_*},H^{s_c,p_*}\cap L_{t,x}^{\frac{d+2}{2}} }\lvert\lvert Z\rvert\rvert_{L_{t,x}^{\frac{d+2}{2}},L_\omega^2\cap L_t^{p_*},H^{s_c,p_*},L_\omega^2},$$ because $\frac{1}{p_*}+\frac{2}{d+2}=\frac{1}{p_*'}=\frac{d+4}{2(d+2)}$.

Then we have that $p_*=\frac{2(d+2)}{d}\leq \frac{d+2}{2}\leq 2(d+2)$, so $\Theta_{V_1} $ is continuously embedded in $ L_{t,x}^{\frac{d+2}{2}}$ and $\Theta_{Z} $ is continuously embedded in $ L_{t,x}^{\frac{d+2}{2}},L_\omega^2$. Thus:

$$\lvert\lvert(\underline{w}*V_1)Z\rvert\rvert_{L^{p_*'},H^{s_c,p_*'},L_\omega^2}\lesssim\lvert\lvert V_1\rvert\rvert_{\Theta_{V_1}}\lvert\lvert Z\rvert\rvert_{\Theta_Z}.$$

And also by Lemma \ref{interpolationargument} and Young's inequality, we have: $$\lvert\lvert(\underline{w}*V_2)Z\rvert\rvert_{L^{p_*'},H^{s_c,p_*'},L_\omega^2}\lesssim \lvert \lvert  V_2\rvert\rvert_{L_t^{\frac{2(d+2)}{d+1}},H_x^{s_c,\frac{2(d+2)}{d+1}}\cap L_{t,x}^{\frac{d+2}{2}}}\lvert\lvert Z\rvert\rvert_{L_{t,x}^{\frac{2(d+2)}{3}},L_\omega^2\cap L_t^{p_*},H_x^{s_c,p_*},L_\omega^2},$$ because $\frac{3}{2(d+2)}+\frac{d+1}{2(d+2)}=\frac{1}{p_*'}=\frac{d+4}{2(d+2)}$.

But because $p_*\leq \frac{2(d+2)}{3}\leq 2(d+2)$, $\Theta_Z$ is continuously embedded in  $ L_{t,x}^{\frac{2(d+2)}{3}},L_\omega^2$, and because $2\leq \frac{d+2}{2}\leq d+2$, $\Theta_{V_2}$ is continuously embedded in $L_{t,x}^{\frac{d+2}{2}}$. Combining the inequality just above and the continuity estimates gives: 
$$\lvert\lvert(\underline{w}*V_2)Z\rvert\rvert_{L^{p_*'},H^{s_c,p_*'},L_\omega^2}\lesssim\lvert\lvert V_2\rvert\rvert_{\Theta_{V_2}}\lvert\lvert Z\rvert\rvert_{\Theta_Z}.$$

Thus, we get: 
$$\lvert\lvert(\underline{w}*V)Z\rvert\rvert_{L^{p_*'},H^{s_c,p_*'},L_\omega^2}\lesssim\lvert\lvert V\rvert\rvert_{\Theta_{V}}\lvert\lvert Z\rvert\rvert_{\Theta_Z}.$$
\textbf{Second estimate} To estimate $\lvert\lvert (\underline{w}*V)Z\rvert\rvert_{L_{t,x}^{\frac{2(d+2)}{d+6}},L_\omega^2}$ we use the Hölder inequality and Young's inequality to get: 

$$\lvert\lvert (\underline{w}*V)Z\rvert\rvert_{L_{t,x}^{\frac{2(d+2)}{d+6}},L_\omega^2}\lesssim \lvert\lvert V \lvert\lvert_{L_{t,x}^2}\lvert\lvert Z\rvert\rvert_{L_{t,x}^{\frac{d+2}{2}},L_\omega^2},$$

because $\frac{d+6}{2(d+2)}=\frac{1}{2}+\frac{2}{d+2}$. 

We have that $\Theta_{V}$ is continuously embedded in $L_{t,x}^2$, and because $p_*=\frac{2(d+2)}{d}\leq \frac{d+2}{2}\leq 2(d+2)$,  $\Theta_{Z} $ is continuously embedded in $ L_{t,x}^{\frac{d+2}{2}},L_\omega^2$. Then: 

$$\lvert\lvert (\underline{w}*V)Z\rvert\rvert_{L_{t,x}^{\frac{2(d+2)}{d+6}},L_\omega^2}\lesssim \lvert\lvert V\rvert\rvert_{\Theta_{V}}\lvert\lvert Z\rvert\rvert_{\Theta_Z}.$$

Combining the inequality above gives the result.

\end{proof}

Before we continue we give an inequality lemma about the product $\circledast$.

\begin{lemma}\label{lemmeconvoldast} Let $s\geq 0$ and $r\geq 1$. For any $p_1,p_2,q_1,q_2\in [1,\infty]$ such that: $$\frac{1}{r}=\frac{1}{p_1}+\frac{1}{q_1}=\frac{1}{p_2}+\frac{1}{q_2}$$ and for $(V_1,V_2)\in L^{p_1}\cap H^{s,p_2}\times L^{q_2}\cap H^{s,q_1}$ we have the following estimate: $$\lvert\lvert w\circledast(V_1,V_2)\rvert\rvert_{H^{s,r}}\lesssim \lvert\lvert V_1\rvert\rvert_{L^{p_1}\cap H^{s,p_2}}\lvert\lvert V_2\rvert\rvert_{L^{q_2}\cap H^{s,q_1}}.$$

\end{lemma}

\begin{proof}
We use an argument of duality by working on the norm $\lvert\lvert\langle \nabla\rangle^s (w\circledast(V_1,V_2))\rvert\rvert_{L^r}$. Let's take $h\in L^{r'}$. 

We have by a change of variables:

$$w\circledast(V_1,V_2)(\tau)=\int w(x_1-x_2,x_2)V_1(x_1+\tau)V_2(x_2+\tau)dx_1dx_2.$$

Then:

\begin{equation}\label{changevar}
    \lvert \int \langle \nabla\rangle^s(w\circledast(V_1,V_2))(\tau)h(\tau)d\tau\rvert \leq \int \lvert w(x_1-x_2,x_2)\langle \nabla\rangle^s(V_1(x_1+\tau)V_2(x_2+\tau))h(\tau)\rvert d\tau dx_1dx_2.
\end{equation}

Using the Hölder inequality on the integral over $\tau$ gives: 

\begin{equation}\label{Holderconvoldast}
    \lvert \int w\circledast(V_1,V_2)(\tau)h(\tau)d\tau\rvert \leq \int \lvert w(x_1-x_2,x_2)\rvert \lvert \lvert x\mapsto \langle \nabla\rangle^s(V_1(x_1+x)V_2(x_2+x))\rvert\rvert_{L^r}dx_1dx_2\lvert\lvert h \rvert\rvert_{L^{r'}}.
\end{equation}

By duality, we already have that: 
$$\lvert\lvert w\circledast(V_1,V_2)\rvert\rvert_{H^{s,r}}\leq \int \lvert w(x_1-x_2,x_2)\rvert \lvert \lvert x\mapsto V_1(x_1+x)V_2(x_2+x)\rvert\rvert_{H^{s,r}}dx_1dx_2.$$

Using Lemma \ref{interpolationargument} gives: 

$$\lvert\lvert w\circledast(V_1,V_2)\rvert\rvert_{L^r}\leq \int \lvert w(x_1-x_2,x_2)\rvert \lvert \lvert x\mapsto V_1(x_1+x)\rvert\rvert_{L^{p_1}\cap H^{s,p_2}}\lvert \lvert x\mapsto V_2(x_2+x)\rvert\rvert_{L^{q_2}\cap H^{s,q_1}}dx_1dx_2.$$

And we get:
$$\lvert\lvert w\circledast(V_1,V_2)\rvert\rvert_{L^r}\leq \int \lvert w(x_1-x_2,x_2)\rvert dx_1dx_2\lvert\lvert V_1\rvert\rvert_{L^{p_1}\cap H^{s,p_2}}\lvert\lvert V_2\rvert\rvert_{L^{q_2}\cap H^{s,q_1}}.$$

The lemma is then proved.
\end{proof}
\bigskip 

\begin{proposition}\label{quad}
    For all $u,v\in \Theta_{V}$,

    \begin{equation}\label{quadZ}
        \lvert\lvert \Tilde{W}_{u,v}(Y)\rvert\rvert_{\Theta_Z}\lesssim \lvert\lvert u\rvert\rvert_{\Theta_{V}}\lvert\lvert v\rvert\rvert_{\Theta_{V}},
    \end{equation} and 

    \begin{equation}\label{quadV}
        \lvert\lvert Re\E(\Tilde{W}_{u,v}(Y))\rvert\rvert_{\Theta_{V_1}}\lesssim \lvert\lvert u\rvert\rvert_{\Theta_{V}}\lvert\lvert v\rvert\rvert_{\Theta_{V}}.
    \end{equation}
\end{proposition}

\begin{proof}
    By Proposition \ref{NLZ}:
$$\lvert\lvert \Tilde{W}_{u,v}(Y)\rvert\rvert_{\Theta_Z}\lesssim \lvert\lvert (w\circledast (u,v))Y\rvert\rvert_{L^{p_*'},H^{s_c,p_*'},L_\omega^2},$$

and by Proposition \ref{NLV1} 
$$\lvert\lvert Re\E(\Bar{Y}\Tilde{W}_{u,v}(Y))\rvert\rvert_{\Theta_{V_1}}\lesssim \lvert\lvert(w\circledast (u,v))Y\rvert\rvert_{L^{p_*'},H^{s_c,p_*'},L_\omega^2}+\lvert\lvert (w\circledast (u,v))Y\rvert\rvert_{L_{t,x}^{\frac{2(d+2)}{d+6}},L_\omega^2}.$$

So, we give estimates on $\lvert\lvert(w\circledast (u,v))Y\rvert\rvert_{L^{p_*'},H^{s_c,p_*'},L_\omega^2}$ and $\lvert\lvert (w\circledast (u,v))Y\rvert\rvert_{L_{t,x}^{\frac{2(d+2)}{d+6}},L_\omega^2}$.

\textbf{First estimate} We begin by giving an estimate of the term $\lvert\lvert (w\circledast (u,v))Y\rvert\rvert_{L^{p_*'},H^{s_c,p_*'},L_\omega^2}$.

By hypothesis on $f$, $\langle\nabla\rangle^{s_c} Y\in L_{t,x}^\infty,L_\omega^2$, so we have by Lemma \ref{interpolationargument}: $$\lvert\lvert (w\circledast (u,v))Y\rvert\rvert_{L^{p_*'},H^{s_c,p_*'},L_\omega^2}\lesssim \lvert\lvert w\circledast (u,v)\rvert\rvert_{L^{p_*'},H^{s_c,p_*'}}.$$   

We denote $u=u_1+u_2\in \Theta_{V_1}+\Theta_{V_2}$ and $v=v_1+v_2\in \Theta_{V_1}+\Theta_{V_2}$ and we have:

\begin{eqnarray*}
    \lvert\lvert w\circledast (u,v)\rvert\rvert_{L^{p_*'},H^{s_c,p_*'},L_\omega^2}&\leq & \lvert\lvert w\circledast (u_1,v_1)\rvert\rvert_{L^{p_*'},H^{s_c,p_*'},L_\omega^2}+\lvert\lvert w\circledast (u_2,v_1)\rvert\rvert_{L^{p_*'},H^{s_c,p_*'},L_\omega^2}\\
    & & +\lvert\lvert w\circledast (u_1,v_2)\rvert\rvert_{L^{p_*'},H^{s_c,p_*'},L_\omega^2}+\lvert\lvert w\circledast (u_2,v_2)\rvert\rvert_{L^{p_*'},H^{s_c,p_*'},L_\omega^2}.
\end{eqnarray*}

Using Lemma \ref{lemmeconvoldast}, we get:

 $$\lvert\lvert w\circledast(u_1,v_1)\rvert\rvert_{L^{p_*'},H_x^{s_c,p_*'}}\lesssim \lvert\lvert u_1\rvert\rvert_{L_t^{p_*},H_x^{s_c,p_*}\cap L_{t,x}^{\frac{d+2}{2}}}\lvert\lvert v_2\rvert\rvert_{L_t^{p_*},H_x^{s_c,p_*}\cap L_{t,x}^{\frac{d+2}{2}}},$$ because $\frac{1}{p_*}+\frac{2}{d+2}=\frac{1}{p_*'}=\frac{d+4}{2(d+2)}$.
 
 And because $\Theta_{V_1}$ is continuously embedded in $L_{t,x}^{\frac{d+2}{2}}$, we get: $$\lvert\lvert w\circledast(u_1,v_1)\rvert\rvert_{L^{p_*'},H_x^{s_c,p_*'}}\lesssim \lvert\lvert u_1\rvert\rvert_{L_t^{p_*},H_x^{s_c,p_*}\cap L_{t,x}^{\frac{d+2}{2}}}\lvert\lvert v_1\rvert\rvert_{L_t^{p_*},H_x^{s_c,p_*}\cap L_{t,x}^{\frac{d+2}{2}}}\leq \lvert\lvert u_1\rvert\rvert_{\Theta_{V_1}}\lvert\lvert v_1\rvert\rvert_{\Theta_{V_1}}.$$

Moreover, by Lemma \ref{lemmeconvoldast} we get: 

$$\lvert\lvert w\circledast(u_1,v_2)\rvert\rvert_{L^{p_*'},H_x^{s_c,p_*'}}\lesssim \lvert\lvert u_1\rvert\rvert_{L_t^{p_*},H_x^{s_c,p_*}\cap L_{t,x}^{\frac{2(d+2)}{3}}}\lvert\lvert v_2\rvert\rvert_{L_{t,x}^{\frac{d+2}{2}}\cap L_t^{\frac{2(d+2)}{d+1}},H_x^{s_c,\frac{2(d+2)}{d+1}}},$$ because $\frac{1}{p_*}+\frac{2}{d+2}=\frac{1}{p_*'}=\frac{d+4}{2(d+2)}$ and $\frac{3}{2(d+2)}+\frac{d+1}{2(d+2)}=\frac{1}{p_*'}=\frac{d+4}{2(d+2)}$. 

But $\Theta_{V_1}$ is continuously embedded in $L_{t,x}^{\frac{2(d+2)}{3}}$, and $\Theta_{V_2}$ is continuously embedded in $L_{t,x}^{\frac{d+2}{2}}$, so:

$$\lvert\lvert w\circledast(u_1,v_2)\rvert\rvert_{L^{p_*'},H_x^{s_c,p_*'}}\lesssim \lvert\lvert u_1\rvert\rvert_{\Theta_{V_1}}\lvert\lvert v_2\rvert\rvert_{\Theta_{V_2}}.$$

By symmetry of the product $(u,v)\to w\circledast(u,v)$ we have similarly: 

$$\lvert\lvert w\circledast(u_2,v_1)\rvert\rvert_{L^{p_*'},H_x^{s_c,p_*'}}\lesssim \lvert\lvert v_1\rvert\rvert_{\Theta_{V_1}}\lvert\lvert u_2\rvert\rvert_{\Theta_{V_2}}.$$

Lastly, using again Lemma \ref{lemmeconvoldast} gives: 

$$\lvert\lvert w\circledast(u_2,v_2)\rvert\rvert_{L^{p_*'},H_x^{s_c,p_*'},L_\omega^2}\lesssim \lvert\lvert u_2\rvert\rvert_{L_t^{\frac{2(d+2)}{d+1}},H_x^{s_c,\frac{2(d+2)}{d+1}}\cap L_{t,x}^{\frac{2(d+2)}{3}}}\lvert\lvert v_2\rvert\rvert_{L_t^{\frac{2(d+2)}{d+1}},H_x^{s_c,\frac{2(d+2)}{d+1}}\cap L_{t,x}^{\frac{2(d+2)}{3}}},$$ because $\frac{d+1}{2(d+2)}+\frac{3}{2(d+2)}=\frac{1}{p_*'}=\frac{d+4}{2(d+2)}$, 

Using that $\Theta_{V_2}$ is continuously embedded in $L_{t,x}^{\frac{2(d+2)}{3}} $, we get: 

$$\lvert\lvert w\circledast(u_2,v_2)\rvert\rvert_{L^{p_*'},H_x^{s_c,p_*'},L_\omega^2}\lesssim \lvert\lvert u_2\rvert\rvert_{\Theta_{V_2}}\lvert\lvert v_2\rvert\rvert_{\Theta_{V_2}}.$$

Thus, by combining the inequalities, we finally have: 
 $$\lvert\lvert (w\circledast (u,v))Y\rvert\rvert_{L^{p_*'},H^{s_c,p_*'},L_\omega^2}\lesssim\lvert\lvert u\rvert\rvert_{\Theta_{V}}\lvert\lvert v\rvert\rvert_{\Theta_{V}}.$$

\textbf{Second estimate} We have that $Y\in L_{t,x}^\infty,L\omega^2$ so:

$$\lvert\lvert (w\circledast (u,v))Y\rvert\rvert_{L_{t,x}^{\frac{2(d+2)}{d+6}},L_\omega^2}\lesssim \lvert\lvert w\circledast (u,v)\rvert\rvert_{L_{t,x}^{\frac{2(d+2)}{d+6}},L_\omega^2}.$$

By Lemma \ref{lemmeconvoldast} we have: 

    $$\lvert\lvert w\circledast(u,v)\rvert\rvert_{L_{t,x}^{\frac{2(d+2)}{d+6}}}\lesssim \lvert\lvert u\rvert\rvert_{L_{t,x}^2}\lvert\lvert v\rvert\rvert_{L_{t,x}^\frac{d+2}{2}}.$$

    Then we use the fact that $\Theta_{V_1}$ and $\Theta_{V_2}$ are continuously embedded in $L_{t,x}^{\frac{d+2}{2}}$ and in $L_{t,x}^2$, and thus $\Theta_{V}$ is continuously embedded in $L_{t,x}^{\frac{d+2}{2}}$ and in $L_{t,x}^2$. We get: 

    $$\lvert\lvert w\circledast(u,v)\rvert\rvert_{L_{t,x}^{\frac{2(d+2)}{d+6}}}\lesssim \lvert\lvert u\rvert\rvert_{\Theta_{V}}\lvert\lvert v\rvert\rvert_{\Theta_{V}}.$$

We conclude by combining the estimates above.

\end{proof}

\subsection{Cubic terms}

We are left with the formally cubic terms of (\ref{nonlin}). Precisely, we give estimates on $\Tilde{W}_{V_1+V_2,V_1+V_2}(Z)$. This is the subject of Proposition \ref{cub2}. But firstly we give the following lemma, useful for the control of those terms. 

\begin{lemma}\label{lemmeconvoldast2}
    Let $s\geq 0$ and $r\geq 1$. Let $r_1,r_2\geq 1$ and $p_i,q_i\geq 1$, $i=1,2,3$, be such that, if $r_3=r_2$:
    $$\frac{1}{r}=\frac{1}{r_i}+\frac{1}{p_i}+\frac{1}{q_i}.$$

    Let $V_1\in H^{s,r_1}\cap L^{r_2}$, $V_2\in L^{p_1}\cap H^{s,p_2}\cap L^{p_3}$ and $Z\in L^{q_1}\cap L^{q_2}\cap H^{s,q_3}$. Then $(w\circledast(V_1,V_2))Z\in H^{s,r}$ and: $$\lvert\lvert (w\circledast(V_1,V_2))Z\rvert\rvert_{H^{s,r}}\lesssim \lvert\lvert V_1\rvert\rvert_{H^{s,r_1}\cap L^{r_2}}\lvert\lvert V_2\rvert\rvert_{L^{p_1}\cap H^{s,p_2}\cap L^{p_3}}\lvert\lvert Z\rvert\rvert_{L^{q_1}\cap L^{q_2}\cap H^{s,q_3}}.$$
\end{lemma}

\begin{proof}
As in the proof of Lemma \ref{lemmeconvoldast} we get by duality that: 

$$N:=\lvert\lvert (w\circledast(V_1,V_2))Z\rvert\rvert_{H^{s,r}}\leq \int \lvert w(x_1-x_2,x_2)\rvert \lvert \lvert x\mapsto V_1(x_1+x)V_2(x_2+x)Z(x)\rvert\rvert_{H^{s,r}}dx_1dx_2.$$

Using Lemma \ref{interpolationargument} gives: 

$$N\leq \int \lvert w(x_1-x_2,x_2)\rvert \lvert \lvert x\mapsto V_1(x_1+x)\rvert\rvert_{H^{s,r_1}\cap L^{r_2}}\lvert \lvert x\mapsto V_2(x_2+x)Z(x)\rvert\rvert_{L^{t_1}\cap H^{s,t_2}}dx_1dx_2,$$

where: $$\frac{1}{t_1}=\frac{1}{p_1}+\frac{1}{q_1},$$ and : $$\frac{1}{t_2}=\frac{1}{p_2}+\frac{1}{q_2}=\frac{1}{p_3}+\frac{1}{q_3}.$$

By Hölder's inequality: 
$$\lvert \lvert x\mapsto V_2(x_2+x)Z(x)\rvert\rvert_{L^{t_1}}\leq \lvert\lvert V_2\rvert\rvert_{L^{p_1}}\lvert\lvert Z\rvert\rvert_{L^{q_1}}.$$

By Lemma \ref{interpolationargument}: 

$$\lvert \lvert x\mapsto V_2(x_2+x)Z(x)\rvert\rvert_{H^{s,t_2}}\lesssim \lvert\lvert V_2\rvert\rvert_{H^{s,p_2}\cap L^{p_3}}\lvert\lvert Z\rvert\rvert_{L^{q_2}\cap H^{s,q_3}}.$$

Then we get:
$$\lvert \lvert x\mapsto V_2(x_2+x)Z(x)\rvert\rvert_{L^{t_1}\cap H^{s,t_2}}\lesssim \lvert\lvert V_2\rvert\rvert_{L^{p_1}\cap H^{s,p_2}\cap L^{p_3}}\lvert\lvert Z\rvert\rvert_{L^{q_1}\cap L^{q_2}\cap H^{s,q_3}}.$$

And finally we have: 
$$\lvert\lvert (w\circledast(V_1,V_2))Z\rvert\rvert_{L^r}\lesssim \int \lvert w(x_1-x_2,x_2)\rvert dx_1dx_2\lvert\lvert V_1\rvert\rvert_{H^{s,r_1}\cap L^{r_2}}\lvert\lvert V_2\rvert\rvert_{L^{p_1}\cap H^{s,p_2}\cap L^{p_3}}\lvert\lvert Z\rvert\rvert_{L^{q_1}\cap L^{q_2}\cap H^{s,q_3}}.$$

The lemma is then proved.
\end{proof}

\begin{remark}
    We note that the product $(u,v)\mapsto w\circledast (u,v)$ is symmetric, so we also have, with the same framework as in Lemma \ref{lemmeconvoldast2}: 
    $$\lvert\lvert (w\circledast(V_2,V_1))Z\rvert\rvert_{H^{s,r}}\lesssim \lvert\lvert V_1\rvert\rvert_{H^{s,r_1}\cap L^{r_2}}\lvert\lvert V_2\rvert\rvert_{L^{p_1}\cap H^{s,p_2}\cap L^{p_3}}\lvert\lvert Z\rvert\rvert_{L^{q_1}\cap L^{q_2}\cap H^{s,q_3}}.$$
\end{remark}

\begin{proposition}\label{cub2}
    For all $Z\in \Theta_Z$ and $u,v\in \Theta_{V}$,

    \begin{equation}\label{cub2Z}
        \lvert\lvert \Tilde{W}_{u,v}(Z)\rvert\rvert_{\Theta_Z}\lesssim \lvert\lvert Z\rvert\rvert_{\Theta_Z}\lvert\lvert u\rvert\rvert_{\Theta_{V}}\lvert\lvert v\rvert\rvert_{\Theta_{V}},
    \end{equation} and 

    \begin{equation}\label{cub2V}
        \lvert\lvert Re\E(\Tilde{W}_{u,v}(Z))\rvert\rvert_{\Theta_{V_1}}\lesssim \lvert\lvert Z\rvert\rvert_{\Theta_Z}\lvert\lvert u\rvert\rvert_{\Theta_{V}}\lvert\lvert v\rvert\rvert_{\Theta_{V}}.
    \end{equation}
\end{proposition}

\begin{proof}
    By Proposition \ref{NLZ}:
$$\lvert\lvert \Tilde{W}_{u,v}(Z)\rvert\rvert_{\Theta_Z}\lesssim \lvert\lvert (w\circledast (u,v))Z\rvert\rvert_{L^{p_*'},H^{s_c,p_*'},L_\omega^2},$$

and by Proposition \ref{NLV1} 
$$\lvert\lvert Re\E(\Bar{Y}\Tilde{W}_{u,v}(Z))\rvert\rvert_{\Theta_{V_1}}\lesssim \lvert\lvert(w\circledast (u,v))Z\rvert\rvert_{L^{p_*'},H^{s_c,p_*'},L_\omega^2}+\lvert\lvert (w\circledast (u,v))Z\rvert\rvert_{L_{t,x}^{\frac{2(d+2)}{d+6}},L_\omega^2}.$$

It is sufficient we give estimates of $\lvert\lvert(w\circledast (u,v))Z\rvert\rvert_{L^{p_*'},H^{s_c,p_*'},L_\omega^2}$ and of $\lvert\lvert (w\circledast (u,v))Z\rvert\rvert_{L_{t,x}^{\frac{2(d+2)}{d+6}},L_\omega^2}$.

\textbf{First estimate} We begin by $\lvert\lvert(w\circledast (u,v))Z\rvert\rvert_{L^{p_*'},H^{s_c,p_*'},L_\omega^2}$, and we denote: $u=u_1+u_2\in \Theta_{V_1}+\Theta_{V_2}$ and $v=v_1+v_2\in \Theta_{V_1}+\Theta_{V_2}$. We have:

\begin{eqnarray*}
    \lvert\lvert (w\circledast (u,v))Z\rvert\rvert_{L^{p_*'},H^{s_c,p_*'},L_\omega^2}&\leq & \lvert\lvert (w\circledast (u_1,v_1))Z\rvert\rvert_{L^{p_*'},H^{s_c,p_*'},L_\omega^2}+\lvert\lvert (w\circledast (u_2,v_1))Z\rvert\rvert_{L^{p_*'},H^{s_c,p_*'},L_\omega^2}\\
    & & +\lvert\lvert (w\circledast (u_1,v_2))Z\rvert\rvert_{L^{p_*'},H^{s_c,p_*'},L_\omega^2}+\lvert\lvert (w\circledast (u_2,v_2))Z\rvert\rvert_{L^{p_*'},H^{s_c,p_*'},L_\omega^2}.
\end{eqnarray*}

By Lemma \ref{interpolationargument} we have: 

$$\lvert\lvert (w\circledast(u_1,v_1))Z\rvert\rvert_{L^{p_*'},H_x^{s_c,p_*'},L_\omega^2}\leq \lvert\lvert w\circledast(u_1,v_1)\rvert\rvert_{L_{t,x}^{\frac{d+2}{2}}\cap L_t^2,H_x^{s_c}}\lvert\lvert Z\rvert\rvert_{L_t^{p_*},H_x^{s_c,p_*},L_\omega^2\cap L_{t,x}^{d+2},L_\omega^2},$$

because: $$\frac{1}{p_*'}=\frac{2}{d+2}+\frac{1}{p_*}=\frac{1}{d+2}+\frac{1}{2}.$$

Moreover, by Lemma \ref{lemmeconvoldast}, we have:

$$\lvert\lvert w\circledast(u_1,v_1)\rvert\rvert_{L_t^2,H_x^{s_c}}\lesssim \lvert\lvert u_1\rvert\rvert_{L_t^{p_*},H_x^{s_c,p_*}\cap L_{t,x}^{d+2}}\lvert\lvert v_1\rvert\rvert_{L_t^{p_*},H_x^{s_c,p_*}\cap L_{t,x}^{d+2}},$$

because: $\frac{1}{2}=\frac{1}{d+2}+\frac{1}{p_*}$.

And, still by Lemma \ref{lemmeconvoldast}: 

$$\lvert\lvert w\circledast(u_1,v_1)\rvert\rvert_{L_{t,x}^{\frac{d+2}{2}}\cap L_t^2,H_x^{s_c}}\lesssim \lvert\lvert u_1\rvert\rvert_{ L_{t,x}^{d+2}}\lvert\lvert v_1\rvert\rvert_{L_{t,x}^{d+2}}.$$

Then we get: 

$$\begin{array}{rl}
     \lvert\lvert (w\circledast(u_1,v_1))Z\rvert\rvert_{L^{p_*'},H_x^{s_c,p_*'},L_\omega^2}& \lesssim \lvert\lvert u_1\rvert\rvert_{L_t^{p_*},H_x^{s_c,p_*}\cap L_{t,x}^{d+2}}\lvert\lvert v_1\rvert\rvert_{L_t^{p_*},H_x^{s_c,p_*}\cap L_{t,x}^{d+2}}\lvert\lvert Z\rvert\rvert_{(L_t^{p_*},H_x^{s_c,p_*}\cap L_{t,x}^{d+2}),L_\omega^2}.
\end{array}$$

 As $p_*\leq d+2 \leq 2(d+2)$, $\Theta_{V_1}$ is continuously embedded in $L_{t,x}^{d+2}$ and $\Theta_Z$ is continuously embedded in $L_{t,x}^{d+2},L_\omega^2$, so: 

 $$\lvert\lvert (w\circledast(u_1,v_1))Z\rvert\rvert_{L^{p_*'},H_x^{s_c,p_*'},L_\omega^2}\lesssim \lvert\lvert u_1\rvert\rvert_{\Theta_{V_1}}\lvert\lvert v_1\rvert\rvert_{\Theta_{V_1}}\lvert\lvert Z\rvert\rvert_{\Theta_Z}$$

Moreover, by Lemma \ref{lemmeconvoldast}: 

$$\lvert\lvert w\circledast(u_2,v_2)\rvert\rvert_{L_{t,x}^{\frac{2(d+2)}{d+3}},H_x^{s_c,\frac{2(d+2)}{d+3}}}\lesssim  \lvert\lvert u_2\rvert\rvert_{L_t^{\frac{2(d+2)}{d+1}},H_x^{s_c,\frac{2(d+2)}{d+1}}\cap L_{t,x}^{d+2}}\lvert\lvert v_2\rvert\rvert_{L_t^{\frac{2(d+2)}{d+1}},H_x^{s_c,\frac{2(d+2)}{d+1}}\cap L_{t,x}^{d+2}},$$

because: $\frac{d+3}{2(d+2)}=\frac{d+1}{2(d+2)}+\frac{1}{d+2}$.

Using Lemma \ref{lemmeconvoldast} again: $$\lvert\lvert w\circledast(u_2,v_2)\rvert\rvert_{L_{t,x}^{\frac{d+2}{2}}}\leq \lvert\lvert u_2\rvert\rvert_{L_{t,x}^{d+2}}\lvert\lvert v_2\rvert\rvert_{L_{t,x}^{d+2}}.$$
 
Finally, we get:
 
$$\lvert\lvert (w\circledast(u_2,v_2))Z\rvert\rvert_{L^{p_*'},H_x^{s_c,p_*'},L_\omega^2}\lesssim \lvert\lvert u_2\rvert\rvert_{L_t^{\frac{2(d+2)}{d+1}},H_x^{s_c,\frac{2(d+2)}{d+1}}\cap L_{t,x}^{d+2}}\lvert\lvert v_2\rvert\rvert_{L_t^{\frac{2(d+2)}{d+1}},H_x^{s_c,\frac{2(d+2)}{d+1}}\cap L_{t,x}^{d+2}}\lvert\lvert Z\rvert\rvert_{(L_{t,x}^{2(d+2)}\cap L_{t}^{p_*},H_x^{s_c,p_*}),L_\omega^2}.$$ 

 Thus: $$\lvert\lvert (w\circledast(u_2,v_2))Z\rvert\rvert_{L^{p_*'},H_x^{s_c,p_*'},L_\omega^2}\lesssim  \lvert\lvert u_2\rvert\rvert_{\Theta_{V_2}}\lvert\lvert v_2\rvert\rvert_{\Theta_{V_2}}\lvert\lvert Z\rvert\rvert_{\Theta_Z}.$$

For the last terms we use again Lemma \ref{lemmeconvoldast2} and if we denote $A:=\lvert\lvert (w\circledast(u_1,v_2))Z\rvert\rvert_{L^{p_*'},H_x^{s_c,p_*'},L_\omega^2}$, we get: 

$$A \lesssim \lvert\lvert u_1\rvert\rvert_{L_t^{p_*},H_x^{s_c,p_*}\cap L_{t,x}^{\frac{4(d+2)}{3}}\cap L_{t,x}^{d+2}}\lvert\lvert v_2\rvert\rvert_{L_{t,x}^{d+2}\cap L_t^{\frac{2(d+2)}{d+1}},H_x^{s_c,\frac{2(d+2)}{d+1}} }\lvert\lvert Z\rvert\rvert_{(L_t^{p_*},H_x^{s_c,p_*}\cap L_{t,x}^{\frac{4(d+2)}{3}}\cap L_{t,x}^{d+2}),L_\omega^2},$$ because: $$\frac{1}{d+2}+\frac{1}{p_*}+\frac{1}{d+2}=\frac{1}{p_*'}=\frac{d+4}{2(d+2)},$$ and: $$\frac{d+1}{2(d+2)}+2\frac{3}{4(d+2)}=\frac{1}{p_*'}=\frac{d+4}{2(d+2)}.$$

Moreover $p_*\leq d+2 \leq 2(d+2)$ so $Z\in L_{t,x}^{d+2},L_\omega^2$ and $u\in L_{t,x}^{d+2}$. 

And finally, $p_*\leq\frac{4}{3} (d+2) \leq 2(d+2)$, so $\Theta_{V_1}$ is continuously embedded in $ L_{t,x}^{\frac{4}{3}(d+2)}$ and $\Theta_{Z}$ is continuously embedded in  $ L_{t,x}^{\frac{4}{3}(d+2)},L_\omega^2$.

So we get:$$\lvert\lvert  (w\circledast (u_1,v_2))Z\rvert\rvert_{L^{p_*'},H_x^{s_c,p_*'},L_\omega^2}\lesssim \lvert\lvert u_1\rvert\rvert_{\Theta_{V_1}}\lvert\lvert v_2\rvert\rvert_{\Theta_{V_2}}\lvert\lvert Z\rvert\rvert_{\Theta_Z}.$$

And by symmetry of the product $(u,v)\to w\circledast(u,v)$, we have similarly: 

$$\lvert\lvert  (w\circledast (u_2,v_1))Z\rvert\rvert_{L^{p_*'},H_x^{s_c,p_*'},L_\omega^2}\lesssim \lvert\lvert v_1\rvert\rvert_{\Theta_{V_1}}\lvert\lvert u_2\rvert\rvert_{\Theta_{V_2}}\lvert\lvert Z\rvert\rvert_{\Theta_Z}.$$

Combining all the inequalities gives: 

$$\lvert\lvert  (w\circledast (u,v))Z\rvert\rvert_{L^{p_*'},H_x^{s_c,p_*'},L_\omega^2}\lesssim \lvert\lvert v\rvert\rvert_{\Theta_{V_1}}\lvert\lvert u\rvert\rvert_{\Theta_{V_2}}\lvert\lvert Z\rvert\rvert_{\Theta_Z}.$$

\textbf{Second estimate} To control $\lvert\lvert (w\circledast (u,v))Z\rvert\rvert_{L_{t,x}^{\frac{2(d+2)}{d+6}},L_\omega^2}$ we use the Hölder inequality and Lemma \ref{lemmeconvoldast} to get: 

    $$\lvert\lvert (w\circledast (u,v))Z)\rvert\rvert_{L_{t,x}^{\frac{2(d+2)}{d+6}},L_\omega^2}\lesssim \lvert\lvert u \lvert\lvert_{L_{t,x}^2}\lvert\lvert v\rvert\rvert_{L_{t,x}^{d+2}}\lvert\lvert Z\rvert\rvert_{L_{t,x}^{d+2},L_\omega^2}.$$

    Then we use the fact that $\Theta_{V_1}$ and $\Theta_{V_2}$ are continuously embedded in $L_{t,x}^2$ and in $L_{t,x}^{d+2}$ because $p_*\leq d+2 \leq 2(d+2)$, so: 

    $$\lvert\lvert (w\circledast(u,v))Z\rvert\rvert_{L_{t,x}^{\frac{2(d+2)}{d+6}},L_\omega^2}\lesssim \lvert\lvert u\rvert\rvert_{\Theta_{V_i}}\lvert\lvert v\rvert\rvert_{\Theta_{V_j}}\lvert\lvert Z\rvert\rvert_{\Theta_Z}.$$

Combining the estimates above allows us to conclude.

\end{proof}

\section{Function space for the initial condition}

In this part, we set a function space for the initial condition, such that we can give estimates of the constant part: 
$$C_{Z_0}=\begin{pmatrix}
    S(t)Z_0\\
    2Re\E[\Bar{Y}S(t)Z_0]\\
    0
    \end{pmatrix}.$$

\begin{proposition}
For all $Z_0\in H^{s_c},L^2_\omega\cap L_x^{\frac{2d}{d+2}},L_\omega^2$, one has that: 

\begin{equation}\label{init}
    \lvert\lvert C_{Z_0}\rvert\rvert_{\Theta_Z\times\Theta_{V_1}\times\Theta_{V_2}}\lesssim \lvert\lvert Z_0\rvert\rvert_{H^{s_c},L^2_\omega} +\lvert\lvert Z_0\rvert\rvert_{L_x^{\frac{2d}{d+2}},L_\omega^2}.
\end{equation}
\end{proposition}

\begin{proof}
By Proposition \ref{Strprop} $S(t)Z_0$ belongs to $\Theta_Z$ and: \begin{equation}\label{initz}
    \lvert\lvert S(t)Z_0\rvert\rvert_{\Theta_Z}\lesssim  \lvert\lvert Z_0\rvert\rvert_{L^2_\omega,H^{s_c}}.
\end{equation}

Now we have to show the $\Theta_{V_1}$ estimates. We start with the space-time Lebesgue norms. We have by Lemma \ref{interpolationargument}: 

\begin{equation}\label{init1}
    \lvert\lvert \E[\Bar{Y}S(t)Z_0]\rvert\rvert_{L_{t,x}^{2(d+2)}\cap L_t^4,H_x^{s_c,2\frac{d}{d-1}}\cap L_t^{2\frac{d+2}{d-1}},H_x^{s_c,2\frac{d+2}{d-1}}}\lesssim\lvert\lvert S(t)Z_0\rvert\rvert_{\Theta_Z}\leq C \lvert\lvert Z_0\rvert\rvert_{L^2_\omega,H^{s_c}},
\end{equation}because  $\langle \nabla\rangle^{s_c} Y\in L_{t,x}^\infty,L_\omega^2$, and $Y\in L_{t,x}^\infty,L_\omega^2$.

For the Besov norm, we use duality. First for the low frequency, let $U\in L^2,B_2^{1/2,0}$ with $$U=\sum_{j\leq0}U_j. $$

We write that by the Hölder inequality: $$\langle U,\E(\Bar{Y}S(t)Z_0)\rangle=\E\bigg(\langle \int_0^\infty S(-t)(\Bar{Y}U)(t)dt,Z_0\rangle\bigg)\leq \lvert\lvert \int_0^\infty S(-t)(\Bar{Y}U)(t)dt\rvert\rvert_{L_x^{\frac{2d}{d-2}}}\lvert\lvert Z_0\rvert\rvert_{L_x^{\frac{2d}{d+2}},L_\omega^2}$$

We use inequality (\ref{prop522}) with $p_1=\infty,q_1=\frac{2d}{d-2}, \sigma_1=1$, and we use the fact that $U$ contains only low frequencies: $$\lvert\lvert \int_0^\infty S(-t)(\Bar{Y}U)(t)dt\rvert\rvert_{L_x^{\frac{2d}{d-2}}}\lesssim \lvert\lvert U\rvert\rvert_{L_t^2,B_2^{\frac{1}{2},1},L_\omega^2}\lesssim \lvert\lvert U\rvert\rvert_{L_t^2,B_2^{\frac{1}{2},0}}.$$

For the high frequency part, we take $U\in L_{t,x}^2$ containing only high frequencies: $U=\sum_{j\geq 0} U_j$.

We claim that, by using Proposition \ref{prop52} with $p_1=\infty$, $q_1=2$, $\sigma_1=0$, the Hölder inequality and the fact that $U$ contains only high frequencies, we have: 

$$\lvert \langle U,\langle\nabla\rangle^{s_c-\frac{1}{2}}\E[\Bar{Y}S(t)Z_0]\rangle\rvert\lesssim \lvert\lvert U\rvert\rvert_{L_{t,x}^2}\lvert\lvert \langle \nabla\rangle^{s_c-\frac{1}{2}}Z_0\rvert\rvert_{L_x^{2},L_\omega^2}.$$

Indeed equation (\ref{prop522}) from Proposition \ref{prop52} gives: $$\lvert\lvert \int_0^\infty S(-t)(\Bar{Y}U)(t)dt\rvert\rvert_{L_x^2,L_\omega^2}\lesssim \lvert\lvert U\rvert\rvert_{L_t^2,B_2^{-\frac{1}{2},-\frac{1}{2}}}\lesssim \lvert\lvert U\rvert\rvert_{L_t^2,B_2^{0,0}}.$$

Combining the inequalities on the high and low frequency parts, one gets that by duality: \begin{equation}\label{init2}
    \lvert\lvert \E(\Bar{Y}S(t)Z_0)\rvert\rvert_{L_t^2,B_2^{-\frac{1}{2},s_c-\frac{1}{2}}}\lesssim \lvert\lvert Z_0\rvert\rvert_{H^{s_c},L^2_\omega}+\lvert\lvert Z_0\rvert\rvert_{L_x^{\frac{2d}{d+2}},L_\omega^2}.
\end{equation}

Combining (\ref{init1}) and (\ref{init2}), one gets: 

\begin{equation}\label{initV}
    \lvert\lvert \E(\Bar{Y}S(t)Z_0)\rvert\rvert_{\Theta_V}\lesssim \lvert\lvert Z_0\rvert\rvert_{H^{s_c},L^2_\omega}+\lvert\lvert Z_0\rvert\rvert_{L_x^{\frac{2d}{d+2}},L_\omega^2}.
\end{equation}

By recalling the expression of $C_{Z_0}$, given by (\ref{constantpart}), we get by combining (\ref{initz}) and (\ref{initV}) the desired result (\ref{init}). 

\end{proof}

\begin{remark}\label{rmkCI}
    With a similar proof we would have: $$\lvert\lvert \E[\Bar{Y}S(t)Z_0]\rvert\rvert_{\Theta_{V_2}}\lesssim \lvert\lvert Z_0\rvert\rvert_{L^{\frac{2d}{d+2}},L_\omega^2}+\lvert\lvert Z_0\rvert\rvert_{H^{s_c+\delta},L_\omega^2},$$
    with $\delta>0$. So $Z_0$ would need to be in a bigger space regarding regularity, if we kept $V_1$ and $V_2$ in a similar functional space, as it was done in \cite{CodS18}.
    This motivates the introduction of the two spaces $\Theta_{V_1}$ and $\Theta_{V_2}$.
\end{remark}

\section{Proof of the main theorem}

In this section, we give the last arguments to finish the proof of Theorem \ref{Thprinc}. The proof of the theorem relies on finding a solution to the fixed point equation (\ref{ptfixe}). We split this section in two subsections. We first write the fixed point argument to solve (\ref{ptfixe}) and then we prove the scattering result.

\subsection{Fixed point argument}

In section 5 we proved that $1-L$ is continuous, invertible of continuous inverse on $\Theta_Z\times \Theta_{V_1}\times \Theta_{V_2}$. Then we can rewrite (\ref{ptfixe}) as: $$\begin{pmatrix} Z\\ V_1 \\ V_2 \end{pmatrix}=(Id-L)^{-1}\bigg( C_{Z_0}+\Rop \begin{pmatrix} Z\\ V_1\\ V_2 \end{pmatrix}\bigg),$$

with: \begin{equation}
    \Rop\begin{pmatrix}
    Z\\ 
    V_1\\
    V_2
    \end{pmatrix}=\begin{pmatrix}
    2\Tilde{m}W_{V_1+V_2}(Z)+\Tilde{W}_{V_1+V_2,V_2+V_2}(Y+Z)\\
    4\Tilde{m}Re\E[\Bar{Y}W_{V_1+V_2}(Z)]+2Re\E[\Bar{Y}\Tilde{W}_{V_1+V_2,V_2+V_2}(Y+Z)]\\
    \E[\lvert Z\rvert^2]
    \end{pmatrix}.
\end{equation}

We use the classical method to solve this Banach fixed point problem. First we define the mapping: 
$$\Phi[Z_0]:\left \{ \begin{array}{rcl}
     \Theta_Z\times\Theta_{V_1}\times\Theta_{V_2} & \to & \Theta_Z\times\Theta_{V_1}\times\Theta_{V_2} \\
     \begin{pmatrix} Z\\ V_1 \\ V_2 \end{pmatrix}& \mapsto & (Id-L)^{-1}\bigg( C_{Z_0}+\Rop \begin{pmatrix} Z\\ V_1\\V_2 \end{pmatrix}\bigg) 
\end{array}\right . .$$

Let us denote $\Theta_0:=H^{s_c},L^2_\omega\cap L_x^{\frac{2d}{d+2}},L_\omega^2$ the space for the initial datum. We are going to show that for $Z_0$ small enough in $\Theta_0$, the mapping $\Phi[Z_0]$ is a contraction on $B(0,R\lvert\lvert Z_0\rvert\rvert_{\Theta_0})$, for some $R>0$. For simplification we denote the norm $\lvert\lvert \cdot \rvert\rvert_{\Theta_Z\times \Theta_{V_1}\times \Theta_{V_2}}$ by $\lvert\lvert \cdot \rvert\rvert$.

By Proposition \ref{proplinpartinv}, we already have that $(Id-L)^{-1}\in\Lop(\Theta_Z\times\Theta_{V_1}\times \Theta_{V_2})$, so: 

$$\bigg\lvert\bigg\lvert \Phi[Z_0]\begin{pmatrix} Z\\ V_1 \\ V_2 \end{pmatrix} \bigg\rvert\bigg\rvert\lesssim \bigg\lvert\bigg\lvert  C_{Z_0}+\Rop \begin{pmatrix} Z\\ V_1 \\ V_2 \end{pmatrix} \bigg\rvert\bigg\rvert.$$

And by the estimate (\ref{init}) we have: $$\lvert\lvert C_{Z_0}\rvert\rvert\lesssim \lvert\lvert Z_0 \rvert\rvert_{\Theta_0}.$$

To study the operator $\Rop$ we will write $\Rop=\Qop_1+\Qop_2$ where $\Qop_1=\Qop_{1,1}+\Qop_{1,2}$ is the quadratic part with: $$\Qop_{1,1}\begin{pmatrix}Z\\ V_1 \\V_2\end{pmatrix}=\begin{pmatrix} 0 \\ 0 \\ \E[\lvert Z\rvert^2]\end{pmatrix}+\begin{pmatrix} 2\Tilde{m} (W_{V_1}(Z)+W_{V_2}(Z))\\ 4\Tilde{m}Re\E[\Bar{Y}(W_{V_1}(Z)+W_{V_2}(Z))] \\ 0 \end{pmatrix},$$ 

and: $$\Qop_{1,2}\begin{pmatrix}Z\\ V_1 \\V_2\end{pmatrix}=\begin{pmatrix} \Tilde{W}_{V_1,V_1}(Y)+2\Tilde{W}_{V_1,V_2}(Y)+\Tilde{W}_{V_2,V_2}(Y)\\ 2Re\E[\Bar{Y}(\Tilde{W}_{V_1,V_1}(Y)+2\Tilde{W}_{V_1,V_2}(Y)+\Tilde{W}_{V_2,V_2}(Y))] \\ 0 \end{pmatrix},$$

and $\Qop_2$ the cubic part: $$\Qop_2\begin{pmatrix}Z\\ V_1 \\V_2\end{pmatrix}=\begin{pmatrix} \Tilde{W}_{V_1,V_1}(Z)+2\Tilde{W}_{V_1,V_2}(Z)+\Tilde{W}_{V_2,V_2}(Z)\\ 2Re\E[\Bar{Y}(\Tilde{W}_{V_1,V_1}(Z)+2\Tilde{W}_{V_1,V_2}(Z)+\Tilde{W}_{V_2,V_2}(Z))] \\ 0 \end{pmatrix},$$

because $(u,v)\mapsto \Tilde{W}_{u,v}(A)$ is symmetric. 

First by Propositions \ref{propembedding} and \ref{wvzprop}: 

$$\bigg\lvert\bigg\lvert \Qop_{1,1}\begin{pmatrix}Z\\ V_1 \\V_2\end{pmatrix}\bigg\rvert\bigg\rvert\lesssim \bigg\lvert \bigg\lvert \begin{pmatrix}Z\\ V_1 \\V_2\end{pmatrix}\bigg\rvert\bigg\rvert^2\leq  R^2\lvert\lvert Z_0\rvert\rvert_{\Theta_0}^2. $$

And we have by bilinearity of $(u,v)\mapsto W_u(v)$:

$$\bigg\lvert \bigg\lvert \Qop_{1,1}\begin{pmatrix}Z\\ V_1 \\V_2\end{pmatrix}-\Qop_{1,1}\begin{pmatrix}Z'\\ V_1 '\\V_2'\end{pmatrix}\bigg\rvert\bigg\rvert\lesssim\bigg( \bigg\lvert \bigg\lvert \begin{pmatrix}Z\\ V_1 \\V_2\end{pmatrix}\bigg\rvert\bigg\rvert+\bigg\lvert \bigg\lvert \begin{pmatrix}Z'\\ V_1' \\V_2'\end{pmatrix}\bigg\rvert\bigg\rvert\bigg)\bigg\lvert \bigg\lvert \begin{pmatrix}Z-Z'\\ V_1-V_1' \\V_2-V_2'\end{pmatrix}\bigg\rvert\bigg\rvert,$$
thus:

$$\bigg\lvert \bigg\lvert \Qop_{1,1}\begin{pmatrix}Z\\ V_1 \\V_2\end{pmatrix}-\Qop_{1,1}\begin{pmatrix}Z'\\ V_1' \\V_2'\end{pmatrix}\bigg\rvert\bigg\rvert\lesssim R\lvert \lvert Z_0\rvert\rvert_{\Theta_0}\bigg\lvert \bigg\lvert \begin{pmatrix}Z-Z'\\ V_1-V_1' \\V_2-V_2'\end{pmatrix}\bigg\rvert\bigg\rvert.$$

From Proposition \ref{quad} we get the estimate:

$$\bigg\lvert\bigg\lvert \Qop_{1,2}\begin{pmatrix}Z\\ V_1 \\V_2\end{pmatrix}\bigg\rvert\bigg\rvert\lesssim \bigg\lvert \bigg\lvert \begin{pmatrix}Z\\ V_1 \\V_2\end{pmatrix}\bigg\rvert\bigg\rvert^2\leq R^2\lvert\lvert Z_0\rvert\rvert_{\Theta_0}^2. $$

And by bilinearity of $(u,v)\mapsto \Tilde{W}_{u,v}(Y)$ we get:

$$\bigg\lvert \bigg\lvert \Qop_{1,2}\begin{pmatrix}Z\\ V_1 \\V_2\end{pmatrix}-\Qop_{1,2}\begin{pmatrix}Z'\\ V_1' \\V_2'\end{pmatrix}\bigg\rvert\bigg\rvert\lesssim\bigg( \bigg\lvert \bigg\lvert \begin{pmatrix}Z\\ V_1 \\V_2\end{pmatrix}\bigg\rvert\bigg\rvert+\bigg\lvert \bigg\lvert \begin{pmatrix}Z'\\ V_1' \\V_2'\end{pmatrix}\bigg\rvert\bigg\rvert\bigg)\bigg\lvert \bigg\lvert \begin{pmatrix}Z-Z'\\ V_1-V_1' \\V_2-V_2'\end{pmatrix}\bigg\rvert\bigg\rvert,$$
thus:

$$\bigg\lvert \bigg\lvert \Qop_{1,2}\begin{pmatrix}Z\\ V_1 \\V_2\end{pmatrix}-\Qop_{1,2}\begin{pmatrix}Z'\\ V_1' \\V_2'\end{pmatrix}\bigg\rvert\bigg\rvert\lesssim R\lvert \lvert Z_0\rvert\rvert_{\Theta_0}\bigg\lvert \bigg\lvert \begin{pmatrix}Z-Z'\\ V_1-V_1' \\V_2-V_2'\end{pmatrix}\bigg\rvert\bigg\rvert.$$

Finally, for the cubic terms we use Proposition \ref{cub2} to get: 

$$\bigg\lvert\bigg\lvert \Qop_2\begin{pmatrix}Z\\ V_1 \\V_2\end{pmatrix}\bigg\rvert\bigg\rvert\lesssim \bigg\lvert \bigg\lvert \begin{pmatrix}Z\\ V_1 \\V_2\end{pmatrix}\bigg\rvert\bigg\rvert^3\leq R^3\lvert\lvert Z_0\rvert\rvert_{\Theta_0}^3. $$

And using that: $$\Tilde{W}_{u,v}(z)-\Tilde{W}_{u',v'}(z')=\Tilde{W}_{u,v-v'}(z)+\Tilde{W}_{u-u',v'}(z)+\Tilde{W}_{u',v'}(z-z'),$$
we obtain: 

$$\bigg\lvert \bigg\lvert \Qop_2\begin{pmatrix}Z\\ V_1 \\V_2\end{pmatrix}-\Qop_2\begin{pmatrix}Z'\\ V_1' \\V_2'\end{pmatrix}\bigg\rvert\bigg\rvert\lesssim\bigg( \bigg\lvert \bigg\lvert \begin{pmatrix}Z\\ V_1 \\V_2\end{pmatrix}\bigg\rvert\bigg\rvert^2+\bigg\lvert \bigg\lvert \begin{pmatrix}Z\\ V_1 \\V_2\end{pmatrix}\bigg\rvert\bigg\rvert\bigg\lvert \bigg\lvert \begin{pmatrix}Z'\\ V_1' \\V_2'\end{pmatrix}\bigg\rvert\bigg\rvert+\bigg\lvert \bigg\lvert \begin{pmatrix}Z'\\ V_1' \\V_2'\end{pmatrix}\bigg\rvert\bigg\rvert^2\bigg)\bigg\lvert \bigg\lvert \begin{pmatrix}Z-Z'\\ V_1-V_1' \\V_2-V_2'\end{pmatrix}\bigg\rvert\bigg\rvert,$$
thus:

$$\bigg\lvert \bigg\lvert \Qop_2\begin{pmatrix}Z\\ V_1 \\V_2\end{pmatrix}-\Qop_2\begin{pmatrix}Z'\\ V_1' \\V_2'\end{pmatrix}\bigg\rvert\bigg\rvert\lesssim R^2\lvert \lvert Z_0\rvert\rvert_{\Theta_0}^2\bigg\lvert \bigg\lvert \begin{pmatrix}Z-Z'\\ V_1-V_1' \\V_2-V_2'\end{pmatrix}\bigg\rvert\bigg\rvert.$$

From the above estimates, one gets that $\Phi[Z_0]$ is indeed a contraction on $B(0,R\lvert\lvert Z_0\rvert\rvert_{\Theta_0})$, for some universal constant $R>0$, for $\lvert\lvert Z_0\rvert\rvert_{\Theta_0}$ small enough. By the Banach's fixed point theorem, we get the existence and uniqueness of a solution to (\ref{eqperturb}) in $B(0,R\lvert\lvert Z_0\rvert\rvert_{\Theta_0})$. 

\subsection{Scattering result}

To get the scattering result we write that: 

$$\begin{array}{rcl}
     Z(t) &= &S(t)\bigg(Z_0-i\int_0^\infty S(-s)(\underline{w}*V(s))Z(s)ds-i \int_0^\infty S(-s)(\underline{w}*V(s))Y(s)ds\\
     & & -i\int_0^\infty S(-s)(w\circledast(V,V))(s)Z(s)ds-i\int_0^\infty S(-s)(w\circledast(V,V))(s)Y(s)ds\bigg)\\ 
     & & +i\int_t^\infty S(t-s)(\underline{w}*V(s))Z(s)ds+i \int_t^\infty S(t-s)(\underline{w}*V(s))Y(s)ds\\
     & & +i\int_t^\infty S(t-s)(w\circledast(V,V))(s)Z(s)ds+i\int_t^\infty S(t-s)(w\circledast(V,V))(s)Y(s)ds,
\end{array}$$

with $V=V_1+V_2$.

By Proposition \ref{wvzprop}, one obtains that $\int_0^\infty S(-s)(\underline{w}*V(s))Z(s)ds\in H^{s_c},L_\omega^2$ and that: $$\lvert\lvert\int_t^\infty S(t-s)(\underline{w}*V(s))Z(s)ds\rvert\rvert_{H^{s_c},L_\omega^2}\lesssim \lvert\lvert Z(s)\rvert\rvert_{\Theta_Z}\lvert\lvert V(s)\mathbf{1}_{s\geq t}\rvert\rvert_{\Theta_V}\to 0 $$

as $t\to +\infty$. Indeed by Proposition \ref{wvzprop} applied to $Z$ and $(V_1+V_2)\mathbf{1}_{s\geq t}$ for $t>0$ we have: 

$$\underset{\tau>0}{sup}\lvert\lvert\int_t^\tau S(t-s)(\underline{w}*V(s))Z(s)ds\rvert\rvert_{H^{s_c},L_\omega^2}\lesssim \lvert\lvert Z(s)\rvert\rvert_{\Theta_Z}\lvert\lvert V(s)\mathbf{1}_{s\geq t}\rvert\rvert_{\Theta_V}.$$

Thus for $\tau>t>0$ we have that: 

$$\lvert\lvert\int_t^\tau S(t-s)(\underline{w}*V(s))Z(s)ds\rvert\rvert_{H^{s_c},L_\omega^2}\lesssim \lvert\lvert Z(s)\rvert\rvert_{\Theta_Z}\lvert\lvert V(s)\mathbf{1}_{s\geq t}\rvert\rvert_{\Theta_V}\to 0,$$

When $t\to +\infty$.We use Cauchy criteria on complete metric space to conclude that $\int_0^\infty S(-s)(\underline{w}*V(s))Z(s)ds\in H^{s_c},L_\omega^2$ and that: $$\lvert\lvert\int_t^\infty S(t-s)(\underline{w}*V(s))Z(s)ds\rvert\rvert_{H^{s_c},L_\omega^2}\to 0 $$

as $t\to +\infty$.

Similarly, by Proposition \ref{proplin}, one obtains that $\int_0^\infty S(-s)(\underline{w}*V(s))Y(s)ds\in H^{s_c},L_\omega^2$ and that: $$\lvert\lvert\int_t^\infty S(t-s)(\underline{w}*V(s))Y(s)ds\rvert\rvert_{H^{s_c},L_\omega^2}\lesssim \lvert\lvert V(s)\mathbf{1}_{s\geq t}\rvert\rvert_{\Theta_V}\to 0 $$

as $t\to +\infty$. 

Moreover, by Proposition \ref{cub2} one obtains that $\int_0^\infty S(-s)(w\circledast(V,V)(s))Z(s)ds\in H^{s_c},L_\omega^2$ and that: $$\lvert\lvert\int_t^\infty S(t-s)(w\circledast(V,V)(s))Z(s)ds\rvert\rvert_{H^{s_c},L_\omega^2}\lesssim \lvert\lvert Z(s)\mathbf{1}_{s\geq t}\rvert\rvert_{\Theta_Z} \lvert\lvert V(s)\mathbf{1}_{s\geq t}\rvert\rvert_{\Theta_V}^2\to 0 $$

as $t\to +\infty$. 

Finally, by Proposition \ref{quad} one obtains that $\int_0^\infty S(-s)(w\circledast(V,V)(s))Y(s)ds\in H^{s_c},L_\omega^2$ and that: $$\lvert\lvert\int_t^\infty S(t-s)(w\circledast(V,V)(s))Y(s)ds\rvert\rvert_{H^{s_c},L_\omega^2}\lesssim \lvert\lvert V(s)\mathbf{1}_{s\geq t}\rvert\rvert_{\Theta_V}^2\to 0 $$

as $t\to +\infty$. Therefore, there exists indeed $Z_\infty\in H^{s_c},L_\omega^2$ such that, as $t\to \infty$: 

$$Z(t)=S(t)Z_\infty+o_{H^{s_c},L_\omega^2}(1).$$

And by the definition of $Z$ we get: 

$$X(t)=Y_f+S(t)Z_\infty+o_{H^{s_c},L_\omega^2}(1).$$

This concludes the proof of Theorem \ref{Thprinc}.

\newpage

\bibliographystyle{amsplain}
\bibliography{biblio}

\end{document}